\documentclass[reqno,11pt]{amsart}
\theoremstyle{plain}
\newtheorem{thm}{Theorem}[section]
\newtheorem{cor}[thm]{Corollary} 
\newtheorem{lemma}[thm]{Lemma} 
\newtheorem{prop}[thm]{Proposition}

\theoremstyle{remark}

\theoremstyle{definition}
\newtheorem{defi}[thm]{Definition}
\newtheorem{example}[thm]{Example}

\newtheorem{remark}[thm]{Remark}

\usepackage{amscd,amssymb,comment,euscript,graphics}
\usepackage{pict2e}
\usepackage{enumitem}
\usepackage[initials]{amsrefs}
\usepackage{color}

\newcount\theTime
\newcount\theHour
\newcount\theMinute
\newcount\theMinuteTens
\newcount\theScratch
\theTime=\number\time
\theHour=\theTime
\divide\theHour by 60
\theScratch=\theHour
\multiply\theScratch by 60
\theMinute=\theTime
\advance\theMinute by -\theScratch
\theMinuteTens=\theMinute
\divide\theMinuteTens by 10
\theScratch=\theMinuteTens
\multiply\theScratch by 10
\advance\theMinute by -\theScratch

\def\today{{\number\day\space
 \ifcase\month\or
  January\or February\or March\or April\or May\or June\or
  July\or August\or September\or October\or November\or December\fi
 \space\number\year}}

\newcommand\Afr{{\mathfrak A}}
\newcommand\alg{{\operatorname{alg}}}
\newcommand\alphah{{\hat{\alpha}}}
\newcommand\at{{\tilde a}}
\newcommand\At{{\widetilde A}}
\newcommand\bh{{\hat b}}
\newcommand\Bt{{\widetilde B}}

\newcommand\clspan{{\overline{\mathrm{span}}\,}}
\newcommand\Cpx{{\mathbf C}}
\newcommand\Ec{{\mathcal{E}}}
\newcommand\Ect{{\widetilde{\Ec}}}
\newcommand\eps{\epsilon}
\newcommand\Et{{\widetilde E}}
\newcommand\Fc{{\mathcal{F}}}
\newcommand\gammah{{\hat{\gamma}}}
\newcommand\Gc{{\mathcal{G}}}
\newcommand\ImagPart{{\mathrm{Im}\,}}
\newcommand\Ints{{\mathbf Z}}
\newcommand\Lc{{\mathcal{L}}}
\newcommand\lspan{\mathrm{span}\,}
\newcommand\Mul{{\operatorname{Mul}}}
\newcommand\Nats{{\mathbf N}}
\newcommand\NC{\operatorname{NC}}
\newcommand\oneh{{\hat 1}}
\newcommand\oup{^{\mathrm o}}
\newcommand\Pc{{\mathcal{P}}}
\newcommand\Pct{{\widetilde\Pc}}
\newcommand\RealPart{{\mathrm{Re}\,}}
\newcommand\Reals{{\mathbf R}}
\newcommand\restrict{{\upharpoonright}}
\newcommand\st{{\tilde s}}
\newcommand\taut{{\tilde\tau}}
\newcommand\Thetat{{\widetilde\Theta}}
\newcommand\yh{{\hat y}}
\newcommand\yt{{\tilde y}}

\begin{document}

\title[R-diagonal]{On algebra-valued R-diagonal elements}

\author[Boedihardjo]{March Boedihardjo$^*$}
\address{M.\ Boedihardjo, Department of Mathematics, Texas A\&M University, College Station, TX 77843-3368, USA}
\email{march@math.tamu.edu}
\thanks{{}$^*$Research supported in part by NSF grant DMS-1301604.}

\author[Dykema]{Ken Dykema$^\dag$}
\address{K.\ Dykema, Department of Mathematics, Texas A\&M University, College Station, TX 77843-3368, USA}
\email{kdykema@math.tamu.edu}
\thanks{{}$^\dag$Research supported in part by NSF grant DMS-1202660.}


\date{April 6, 2016}

\begin{abstract}
For an element in an algebra-valued $*$-noncommutative probability space,
equivalent conditions for algebra-valued R-diagonality (a notion introduced by {\'S}niady and Speicher)
are proved.
Formal power series relations involving the moments and cumulants of such R-diagonal elements are proved.
Decompositions of algebra-valued R-diagonal elements into products of the form unitary times self-adjoint are investigated;
sufficient conditions, in terms of cumulants, for $*$-freeness of the unitary and the self-adjoint part are proved,
and a tracial example is given where $*$-freeness fails.
The particular case of algebra-valued circular elements is considered;
as an application, the polar decompostion of
the quasinilpotent DT-operator is described.
\end{abstract}

\maketitle

\section{Introduction}

Let $B$ be a unital algebra over the complex numbers.
A {$B$-valued noncommutative probability space} is a pair $(A,\Ec)$, where $A$ is a unital algebra containing a unitally embedded
copy of $B$ and $\Ec$ is a {conditional expectation}, namely, an idempotent linear mapping
$\Ec:A\to B$ that restricts to the identity on $B$
and satisfies $\Ec(b_1ab_2)=b_1\Ec(a)b_2$ for every $a\in A$ and $b_1,b_2\in B$.
Elements of $A$ are called {random variables} or $B$-valued random variables.
When $B$ and $A$ are $*$-algebras and $\Ec$ is $*$-preserving, then the pair $(A,\Ec)$ is called a 
{$B$-valued $*$-noncommutative probability space.}
The $B$-valued $*$-distribution of an element $a\in A$ is, loosely speaking, the collection of $*$-moments of the form
$\Ec(a^{\eps(1)}b_1a^{\eps(2)}\cdots b_{n-1}a^{\eps(n)})$ for $n\ge1$, $\eps(1),\ldots,\eps(n)\in\{1,*\}$ and $b_1,\ldots,b_{n-1}\in B$.
See Definition~\ref{def:*distr} for a formal version.

We study $B$-valued R-diagonal random variables
in $B$-valued $*$-non\-com\-mu\-ta\-tive probability spaces.
Our motivation is the results of \cite{BD}, where certain random matrices are proved to be asymptotically $B$-valued
R-diagonal.
In the scalar-valued case ($B=\Cpx)$, R-diagonal random variables were introduced by Nica and Speicher \cite{NS97}
and these include many natural and important examples in free probability theory.
In a subsequent paper \cite{NSS01}, Nica, Shlyakhtenko and Speicher found several equivalent characterizations
of scalar-valued R-diagonal elements.
In~\cite{SS01}, {\'S}niady and Speicher introduced $B$-valued R-diagonal elements and proved some equivalent characterizations of them.
In this paper we firstly prove some further
characterizations of $B$-valued R-diagonal elements, which are analogous to those of \cite{NSS01}.
Secondly, we prove a result involving power series for $B$-valued R-diagonal elements that is similar to the
power series relation between the R-transform and the the moment series of a single random variable.
Thirdly, we examine polar decompositions of $B$-valued R-diagonal elements, which is a more delicate topic than in the scalar-valued case.
We also study $B$-valued circular elements (a special
case of $B$-valued R-diagonal elements)
and we prove a new result about the polar decomposition of a quasinilpotent DT-operator.

\begin{defi}\label{def:maxaltptn}
Given $n\in\Nats$ and $\eps=(\eps(1),\ldots,\eps(n))\in\{1,*\}^n$,
we define the {\em maximal alternating interval partition} $\sigma(\eps)$ to be the interval partition
of $\{1,\ldots,n\}$ whose blocks $V$ are the maximal interval subsets of $\{1,\ldots,n\}$ such that if $j\in V$ and $j+1\in V$,
then $\eps(j)\ne\eps(j+1)$.
\end{defi}

The following definition is
a reformulation of one of the characterizations of R-diagonality found in~\cite{SS01} and (in the scalar-valued case) in~\cite{NSS01}.
\begin{defi}\label{def:rdiag}
We say that an element $a$ in a $B$-valued $*$-noncommutative probability space $(A,\Ec)$ is {\em $B$-valued R-diagonal}
if for every integer $k\ge0$ and every $b_1,\ldots,b_{2k}\in B$ we have
\[
\Ec(ab_1a^*b_2ab_3a^*\cdots b_{2k-2}ab_{2k-1}a^*b_{2k}a)=0,
\]
(namely, odd alternating moments vanish)
and, for every integer $n\ge1$, every $\eps\in\{1,*\}^n$ and every choice of $b_1,b_2,\ldots b_n\in B$, we have
\begin{equation}\label{eq:IPcond}
\Ec\left(\prod_{V\in\sigma(\eps)}\left(\left(\prod_{j\in V}a^{\eps(j)}b_j\right)-\Ec\left(\prod_{j\in V}a^{\eps(j)}b_j\right)\right)\right)=0,
\end{equation}
where in each of the three products above, the terms are taken in order of increasing indices.
\end{defi}

\begin{remark}\label{rem:a*Rdiag}
Clearly $a$ is $B$-valued R-diagonal if and only if $a^*$ is $B$-valued R-diagonal.
\end{remark}

\begin{remark}\label{rem:l2.1}
It is not difficult to see, by expansion of the left-hand-side of~\eqref{eq:IPcond} and induction on $n$,
that the $B$-valued $*$-distribution
of a $B$-valued R-diagonal element is completely determined by the collection of
even, alternating moments, namely, those of
of the form
\[
\Ec(a^*b_1ab_2a^*b_3a\cdots b_{2k-2}a^*b_{2k-1}a)\quad\text{and}\quad\Ec(ab_1a^*b_2ab_3a^*\cdots b_{2k-2}ab_{2k-1}a^*).
\]
\end{remark}

In Theorem~\ref{thm:rdiagEq}, we prove equivalence of seven conditions for a $B$-valued random variable
$a$, including the condition for R-diagonalality in Definition~\ref{def:rdiag}.
(For some of these equivalences we refer to~\cite{SS01}.)
The six others are, in fact, $B$-valued analogues of those found in \cite{NSS01} for the case $B=\Cpx$.
One of these, condition~\ref{it:cum}, is in terms of Speicher's noncrossing cumulants for the pair $(a,a^*)$,
namely, that only those associated with even alternating sequences in $a$ and $a^*$ may be nonvanishing.
Another, condition~\ref{it:M2} of Theorem~\ref{thm:rdiagEq},
is that the matrix $\left(\begin{smallmatrix}0&a\\a^*&0\end{smallmatrix}\right)$ is free from $M_2(B)$
with amalgamation over the diagonal matrices with entries in $B$,
while still another, condition~\ref{it:Pmoms}, is an easy reformulation of Definition~\ref{def:rdiag}.
The proofs of the various equivalences are similar to those found in \cite{NSS01}.

A unitary element of a $B$-valued $*$-noncommutative probability space $(A,\Ec)$ is, of course, an element $u\in A$
such that $u^*u=uu^*=1$.
A {\em Haar unitary element} is a unitary element satisfying $\Ec(u^k)=0$ for all integers $k\ge1$.
In the tracial, scalar-valued case, it is well known (see Proposition 2.6 of \cite{NSS01})
that being $R$-diagonal is equivalent to having the same $*$-distribution
as an element of the form $up$, where $u$ is Haar unitary and $p$ is self-adjoint
and such that $\{u,u^*\}$ and $\{p\}$ are free.
(In the case of a C$^*$-noncommutative probability space, one can also take $p\ge0$
--- this is well known, or see, for example, Corollary~\ref{cor:scalarpolar} for a proof.)
The analogous statement is not true in the tracial, algebra-valued case.
Example~\ref{ex:nofreepolar} provides a counter example.
However, in Proposition~\ref{prop:HUtheta} and Theorem~\ref{thm:C*pos},
we do characterize, in terms of cumulants, when an algebra-valued R-diagonal element has the same $B$-valued $*$-distribution
as a product $up$ with $p$ self-adjoint, with $u$ a $B$-normalizing Haar unitary element and with $\{u,u^*\}$ and $\{p\}$ free over $B$.
As an application, Corollary~\ref{cor:DTpolar} shows that the polar decomposition of the quasinilpotent DT-operator has this form.

The contents of the rest of the paper are as follows:
Section~\ref{sec:cums} briefly recalls the formulation from~\cite{NSS02} of Speicher's $B$-valued cumlants, introduces notation and proves
some straightforward results about traces and
about self-adjointness of $B$-valued distributions and cumulants.
Section~\ref{sec:Rdiag} deals with the equivalence of the seven conditions that characterize algebra-valued R-diagonal elements.
Section~\ref{sec:series} establishes formal power series relations involving $B$-valued alternating moments and $B$-valued alternating cumulants
of an R-diagonal element.
Section~\ref{sec:polar} proves conditions for traciality of $*$--distributions of algebra-valued R-diagonal elements and examines polar decompositions
and the like for R-diagonal elements.
Section~\ref{sec:circ} examines algebra-valued circular elements, which comprise a special case of algebra-valued R-diagonal elements;
the notion of these has appeared before, notably in work of {\'S}niady \cite{Sn03}.
This section also contains Example~\ref{ex:nofreepolar} concerning the polar decomposition, mentioned above.
Finally, in an appendix, we investigate the distribution (with respect to a trace) of the positive part of the operator in this example.

\section{Cumulants, traces, and $*$-distributions}
\label{sec:cums}

In this section, we 
briefly recall a formulation (from~\cite{NSS02}) of Speicher's theory~\cite{Sp98} of $B$-valued cumulants
and describe the notation we will use.
We also prove some straightforward results about traciality and self-adjointness.

Given a family $(a_i)_{i\in I}$ of random variables in a $B$-valued noncommutative probability space $(A,\Ec)$,
and given $j=(j(1),\ldots,j(n))\in\bigcup_{n\ge1}I^n$,
the family's {\em $B$-valued distribution} is the $B$-bimodular map $\Theta:B\langle (X_i)_{i\in I}\rangle\to B$ given by
\begin{equation}\label{eq:distr}
\Theta(b_0X_{i_1}b_1\cdots X_{i_n}b_n)=\Ec(b_0a_{i_1}b_1\cdots a_{i_n}b_n),
\end{equation}
where the domain of $\Theta$ is the universal algebra (over $\Cpx$) generated by a copy of $B$ and indeterminants $X_i$;
more formally, it is the algebraic free product (with amalgamtion over $\Cpx)$ of a copy of $B$
and copies of the polynomial algebra $\Cpx[X_i]$ with indeterminant $X_i$.
The corresponding {\em cumulant map} is a $\Cpx$-multilinear
map $\alpha_j:B^{n-1}\to B$.  These are defined, recursively, by the moment-cumulant formula
\begin{equation}\label{eq:alphacummom}
\Ec(a_{j(1)}b_1a_{j(2)}\cdots b_{n-1}a_{j(n)})=\sum_{\pi\in\NC(n)}\alphah_j(\pi)[b_1,\ldots,b_{n-1}],
\end{equation}
where $\NC(n)$ is the set of all noncrossing partitions of $\{1,\ldots,n\}$ and where for $\pi\in\NC(n)$,
$\alphah_j(\pi)$ is a multilinear map defined in terms of the cumulant maps $\alpha_{j'}$ for the $j'$ obtained by restricting
$j$ to the blocks of $\pi$.
In detail,
$\alphah_j(\pi)$ can be specified
(recursively) as follows.
If $\pi=1_n$, then
\begin{equation*}
\alphah_j(\pi)[b_1,\ldots,b_{n-1}]=\alpha_j(b_1,\ldots,b_{n-1}),
\end{equation*}
while if $\pi\ne1_n$ then, selecting an interval block $\{p,p+1,\ldots,p+q-1\}\in\pi$ with $p\ge1$ and $q\ge1$,
letting $\pi'\in\NC(n-q)$ be obtained by restricting $\pi$ to $\{1,\ldots,p-1\}\cup\{p+q,\ldots,n\}$ and renumbering to preserve order,
and letting $j'\in I^{n-q}$ and $j''\in I^q$ be
\begin{align}
j'&=(j(1),j(2),\ldots,j(p-1),\,j(p+q),j(p+q+1),\ldots,j(n)), \label{eq:j'} \\
j''&=(j(p),j(p+1),\ldots,j(p+q-1)), \label{eq:j''}
\end{align}
we have
\begin{multline}\label{eq:cumrec}
\alphah_j(\pi)[b_1,\ldots,b_{n-1}]= \\
=\begin{cases}
\begin{aligned}[b]
\alphah_{j'}(\pi')[b_1,\ldots,&b_{p-2}, \\
b_{p-1}&\alpha_{j''}(b_p,\ldots,b_{p+q-2})b_{p+q-1}, \\
&\hspace{7em} b_{p+q},\ldots,b_{n-1}],
\end{aligned}
&p\ge2,\,p+q-1<n \\
\alphah_{j'}(\pi')[b_1,\ldots,b_{p-2}]b_{p-1}\alpha_{j''}(b_p,\ldots,b_{n-1}),&p\ge2,\,p+q-1=n, \\
\alpha_{j''}(b_1,\ldots,b_{q-1})b_q\alphah_{j'}(\pi')[b_{q+1},\ldots,b_{n-1}],&p=1,\,q<n.
\end{cases}
\end{multline}

We will use the notation $\psi_j$ for the multilinear moment map
\begin{equation}\label{eq:psij}
\psi_j(b_1,\ldots,b_{n-1})=\Ec(a_{j(1)}b_1a_{j(2)}\cdots b_{n-1}a_{j(n)}).
\end{equation}

Given a tracial linear functional on $B$ and a $B$-valued noncommutative probability space $(A,\Ec)$,
we will now characterize, in terms of $B$-valued cumulants, when the composition $\tau\circ\Ec$ is tracial.

In the following lemma and its proof, we will use the notation $c$ for cyclic left permutations.
In particular, 
\[
\text{if }j=(j(1),\ldots,j(n))\in I^n\text{ then }c(j)=(j(2),j(3),\ldots,j(n),j(1)),
\]
and if $\pi\in\NC(n)$ for some $n$, then $c(\pi)\in\NC(n)$ is the permutation obtained from $\pi$ by applying the mapping
\[
j\mapsto\begin{cases} n,&j=1 \\j-1,&j>1\end{cases}
\]
to the underlying set $\{1,\ldots,n\}$;
so, for example, if $\pi=\{\{1,2\},\{3,4,5\}\}\in\NC(5)$, then $c(\pi)=\{\{1,5\},\{2,3,4\}\}$.

\begin{lemma}\label{lem:taualphapi}
Consider the $B$-valued distribution of a family $(a_i)_{i\in I}$ and the corresponding cumulant maps $\alpha_j$.
Suppose $\tau$ is a tracial linear functional on $B$.
Fix $n\ge2$ and suppose that for all $m\in\{1,2,\ldots,n-1\}$, all $j'\in I^m$ and all $b_1,\ldots,b_m\in B$,
we have 
\begin{equation}\label{eq:taualphaassump}
\tau(\alpha_{j'}(b_1,\ldots,b_{m-1})b_m)=\tau(b_1\alpha_{c(j')}(b_2,\ldots,b_m)).
\end{equation}
Then for every $\pi\in\NC(n)\backslash\{1_n\}$, $j\in I^n$ and $b_1,\ldots,b_n\in B$, we have
\begin{equation}\label{eq:taualphapi}
\tau(\alphah_j(\pi)[b_1,\ldots,b_{n-1}]b_n)=\tau(\alphah_{c(j)}(c(\pi))[b_2,\ldots,b_n]b_1).
\end{equation}
\end{lemma}
\begin{proof}
We proceed by induction on $n$.
To begin, if $n=2$, then $\pi=\{\{1\},\{2\}\}$ and using~\eqref{eq:cumrec} we have
\[
\tau(\alphah_j(\pi)[b_1]b_2)=\tau(\alpha_{(j(1))}b_1\alpha_{(j(2))}b_2)
=\tau(b_1\alpha_{(j(2))}b_2\alpha_{(j(1))})=\tau(b_1\alphah_{c(j)}(\pi)[b_2]),
\]
as required.
Assume $n\ge3$.
By the induction hypothesis, for every $m\in\{1,\ldots,n-1\}$, $j'\in I^m$ and $\pi'\in\NC(m)$, including,
by the original hypothesis~\eqref{eq:taualphaassump},
the case $\pi'=1_m$, we have
\begin{equation}\label{eq:taualphapi'}
\tau(\alphah_{j'}(\pi')[b_1,\ldots,b_{m-1}]b_m)=\tau(\alphah_{c(j')}(c(\pi'))[b_2,\ldots,b_m]b_1).
\end{equation}
Since $\pi\ne1_n$, there is an interval block $\{p,p+1,\ldots,p+q-1\}\in\pi$ with $p\ge2$ and $q\ge1$.
Let $j'$ and $j''$ be as in \eqref{eq:j'}--\eqref{eq:j''} and let $\pi'\in\NC(n-q)$ be as described above those equations.
If $p+q-1<n$ and $p\ge3$, then 
\begin{align*}
\tau(\alphah_j(&\pi)[b_1,\ldots,b_{n-1}]b_n) \\
&=\tau(\alphah_{j'}(\pi')[b_1,\ldots,b_{p-2},b_{p-1}\alpha_{j''}(b_p,\ldots,b_{p+q-2})b_{p+q-1},b_{p+q},\ldots,b_{n-1}]b_n) \\
&=\tau(\alphah_{c(j')}(c(\pi'))[b_2,\ldots,b_{p-2},b_{p-1}\alpha_{j''}(b_p,\ldots,b_{p+q-2})b_{p+q-1},b_{p+q},\ldots,b_n]b_1) \\
&=\tau(\alphah_{c(j)}(c(\pi))[b_2,\ldots,b_n]b_1),
\end{align*}
where we have used, respectively, the first case on the right-hand-side of~\eqref{eq:cumrec}, \eqref{eq:taualphapi'}
and again the first case on the right-hand-side of~\eqref{eq:cumrec}.
If $p+q-1<n$ and $p=2$, then
\begin{align*}
\tau(\alphah_j(\pi)[b_1,\ldots,b_{n-1}]b_n)&=\tau(\alphah_{j'}(\pi')[b_1\alpha_{j''}(b_2,\ldots,b_{q})b_{q+1},b_{q+2},\ldots,b_{n-1}]b_n) \\
&=\tau(\alphah_{c(j')}(c(\pi'))[b_{q+2},\ldots,b_n]b_1\alpha_{j''}(b_2,\ldots,b_{q})b_{q+1}) \displaybreak[2] \\
&=\tau(\alpha_{j''}(b_2,\ldots,b_{q})b_{q+1}\alphah_{c(j')}(c(\pi'))[b_{q+2},\ldots,b_n]b_1) \\
&=\tau(\alphah_{c(j)}(c(\pi))[b_2,\ldots,b_n]b_1),
\end{align*}
where we have used, respectively, the first case on the right-hand-side of~\eqref{eq:cumrec}, \eqref{eq:taualphapi'},
the trace property of $\tau$
and the third case on the right-hand-side of~\eqref{eq:cumrec}.
The case of $p\ge3$ and $p+q-1=n$ is done similarly, while if $p=2$ and $q=n-1$, then
\begin{align*}
\tau(\alphah_j(\pi)[b_1,\ldots,b_{n-1}]b_n)&=\tau(\alpha_{j(1)}b_1\alpha_{j''}(b_2,\ldots,b_{n-1})b_n) \\
&=\tau(\alpha_{j''}(b_2,\ldots,b_{n-1})b_n\alpha_{j(1)}b_1) \\
&=\tau(\alphah_{c(j)}(c(\pi))[b_2,\ldots,b_n]b_1),
\end{align*}
where we have used, respecively,
the second case on the right-hand-side of~\eqref{eq:cumrec},
the trace property of $\tau$
and the third case on the right-hand-side of~\eqref{eq:cumrec}.
Thus,~\eqref{eq:taualphapi} holds and the lemma is proved.
\end{proof}

\begin{prop}\label{prop:trace}
Let $\Theta$ be the $B$-valued distribution of a family $(a_i)_{i\in I}$ with corresponding moment maps $\psi_j$
and cumulant maps $\alpha_j$.
Suppose $\tau$ is a tracial linear functional on $B$.
Then the following are equivalent:
\begin{enumerate}[label=(\roman*),labelwidth=3ex,leftmargin=30pt]
\item\label{it:tracial} the linear functional $\tau\circ\Theta$ on $B\langle X_i\mid i\in I\rangle$ is tracial,
\item\label{it:psitrace} $\forall n\ge2$ $\forall j\in I^n$ $\forall b_1,\ldots,b_n\in B$, we have
\begin{equation}\label{eq:taupsi}
\tau(\psi_j(b_1,\ldots,b_{n-1})b_n)=\tau(b_1\psi_{c(j)}(b_2,\ldots,b_n)),
\end{equation}
\item\label{it:alphatrace} $\forall n\ge1$ $\forall j\in I^n$ $\forall b_1,\ldots,b_n\in B$, we have
\begin{equation}\label{eq:taualpha}
\tau(\alpha_j(b_1,\ldots,b_{n-1})b_n)=\tau(b_1\alpha_{c(j)}(b_2,\ldots,b_n)).
\end{equation}
\end{enumerate}
\end{prop}
\begin{proof}
The equivalence of~\ref{it:tracial} and~\ref{it:psitrace} is easily seen from the definition~\eqref{eq:psij} of $\psi_j$.

The proof of \ref{it:alphatrace} $\Longrightarrow$ \ref{it:psitrace} follows from the moment-cumulant
formula~\eqref{eq:alphacummom} and Lemma \ref{lem:taualphapi}.

We will prove \ref{it:psitrace} $\Longrightarrow$ \ref{it:alphatrace} using the moment-cumulant
formula~\eqref{eq:alphacummom} and Lemma~\ref{lem:taualphapi}.
Suppose \ref{it:psitrace} holds and let us show~\eqref{eq:taualpha} holds by induction on $n\ge1$.
The case $n=1$ is from the tracial property of $\tau$.
Fix $n_0\ge2$ and suppose~\eqref{eq:taualpha} holds for all $n<n_0$.
Then, by Lemma~\ref{lem:taualphapi}, for all $\pi\in\NC(n_0)\backslash\{1_{n_0}\}$, all $j\in I^{n_0}$ and all $b_1,\ldots,b_{n_0}\in B$, we have
\[
\tau(\alphah_j(\pi)[b_1,\ldots,b_{n_0-1}]b_{n_0})=\tau(\alphah_{c(j)}(c(\pi))[b_2,\ldots,b_{n_0}]b_1).
\]
Combining this with the moment-cumulant formula~\eqref{eq:alphacummom} and using~\eqref{eq:taupsi}, we get
\begin{align*}
\tau(\alpha_j(b_1,\ldots&,b_{n_0-1})b_{n_0}) \\
&=\tau(\psi_j(b_1,\ldots,b_{n_0-1})b_{n_0})
-\sum_{\pi\in\NC(n_0)\backslash\{1_{n_0}\}}\tau(\alphah_j(\pi)[b_1,\ldots,b_{n_0-1}]b_{n_0}) \\
&=\tau(\psi_{c(j)}(b_2,\ldots,b_{n_0})b_1)
-\sum_{\pi\in\NC(n_0)\backslash\{1_{n_0}\}}\tau(\alphah_{c(j)}(c(\pi))[b_2,\ldots,b_{n_0}]b_1) \\
&=\tau(\alpha_{c(j)}(b_2,\ldots,b_{n_0})b_1),
\end{align*}
as required.
\end{proof}

We now turn to questions of self-adjointness.
Suppose $B$ is a $*$-algebra and
consider a family $(a_i)_{i\in I}$ of $B$-valued random variables in a $B$-valued noncommutative probability space $(A,\Ec)$.
Consider an involution $s:I\to I$.
Let $B\langle X_i\mid i\in I\rangle$ be endowed with the $*$-algebra structure
coming from the $*$-operation on $B$ and by setting $X_i^*=X_{s(i)}$ for all $i\in I$.
For each $n\ge1$, let $\st:I^n\to I^n$ be defined by 
\[
\st((j(1),j(2),\ldots,j(n))=(s(j(n)),\ldots,s(j(2)),s(j(1))).
\]

\begin{lemma}\label{lem:alphapi*}
Let $\alpha_j$ be the cumulant maps of the family $(a_i)_{i\in I}$.
Fix $n\ge2$ and suppose that for all $m\in\{1,2,\ldots,n-1\}$, all $j'\in I^m$
and all $b_1,\ldots,b_{m-1}\in B$, we have
\[
\alpha_{j'}(b_1,\ldots,b_{m-1})^*=\alpha_{\st(j')}(b_{m-1}^*,\ldots,b_1^*).
\]
Then for all $\pi\in\NC(n)\backslash\{1_n\}$, all $j\in I^n$ and all $b_1,\ldots,b_{n-1}\in B$, we have
\[
\alphah_j(\pi)[b_1,\ldots,b_{n-1}]^*=\alpha_j(r(\pi))[b_{n-1}^*,\ldots,b_1^*],
\]
where $r(\pi)$ is the noncrossing partition obtained from $\pi$ by applying the reflection $j\mapsto n-j$ to
all elements in the underlying set $\{1,\ldots,n-1\}$.
\end{lemma}
\begin{proof}
This follows in a straightforward manner by induction on the number of blocks in $\pi$, from the recursion formula~\eqref{eq:cumrec}
for cumulants.
\end{proof}

The following proposition gives a criterion for self-adjointness of the distribution of a family of $B$-valued random variables
in terms of properties of the $B$-valued cumulant maps.
\begin{prop}\label{prop:alpha*}
Let $\Theta$ be the $B$-valued distribution of a family $(a_i)_{i\in I}$ with corresponding moment maps $\psi_j$
and cumulant maps $\alpha_j$.
The following are equivalent:
\begin{enumerate}[label=(\roman*),labelwidth=3ex,leftmargin=30pt]
\item\label{it:Thetasa} the linear mapping $\Theta:B\langle X_i\mid i\in I\rangle\to B$ is self-adjoint.
\item\label{it:psi*} $\forall n\ge1$ $\forall j\in I^n$ $\forall b_1,\ldots,b_{n-1}\in B$, we have
\[
\psi_j(b_1,\ldots,b_{n-1})^*=\psi_{\st(j)}(b_{n-1}^*,\ldots,b_1^*),
\]
\item\label{it:alpha*} $\forall n\ge1$ $\forall j\in I^n$ $\forall b_1,\ldots,b_n\in B$, we have
\[
\alpha_j(b_1,\ldots,b_{n-1})^*=\alpha_{\st(j)}(b_{n-1}^*,\ldots,b_1^*).
\]
\end{enumerate}
Moreover, if $(A,\Ec)$ is a $*$-noncommutative probability space and if
$a_i^*=a_{s(i)}$ for all $i\in I$, then the above conditions are satisfied.
\end{prop}

\begin{defi}\label{def:*distr}
If the conditions in the last sentence of Proposition~\ref{prop:alpha*} hold, then we say $\Theta$ is the {\em $B$-valued $*$-distribution}
of the family $(a_i)_{i\in I}$.
\end{defi}

\begin{proof}[Proof of Proposition~\ref{prop:alpha*}.]
The equivalence of~\ref{it:Thetasa} and~\ref{it:psi*} follows directly from the definitions.
The equivalence of~\ref{it:psi*} and~\ref{it:alpha*} is a straightforward application of the moment-cumulant formula
and Lemma~\ref{lem:alphapi*}.
\end{proof}

We now suppose that $B$ is a C$^*$-algebra and $\Theta$ is
a $B$-valued $*$-distribution, as considered above.
We say that $\Theta$ is {\em positive} if $\Theta(p^*p)\ge0$ for every $p\in B\langle X_i\mid i\in I\rangle$.
It can be difficult to verify positivity of a $*$-distribution $\Theta$ only from knowing the cumulants,
though there are special cases that are exceptions to this statement (for example, the $B$-valued semicircular
and $B$-valued circular elements, discussed in Section~\ref{sec:circ}).

\section{Algebra-valued R-diagonal elements}
\label{sec:Rdiag}

Let $B$ be a unital $*$-algebra and let $(A,\Ec)$ be a $B$-valued $*$-noncommutative probability space.

It is convenient to introduce here some notation we will use below.
By an {\em enlargement} of $(A,\Ec)$, we will mean a $B$-valued $*$-noncommutative probability space $(\At,\Ect)$
with an embedding $A\hookrightarrow\At$ so that the diagram 
\[
\begin{matrix}
A&\hookrightarrow&\At \\
\cup&&\cup \\
B&=&B
\end{matrix}
\]
commutes and $\Ect\restrict_A=\Ec$.

For an element $a\in A$, 
consider the following sets of words and their centerings, formed from alternating $a$ and $a^*$, with elements of $B$ between:
\begin{align}
\Pc_{11}&=\{b_0ab_1a^*b_2ab_3a^*b_4\cdots ab_{2k-1}a^*b_{2k}ab_{2k+1}\mid k\ge0,\,b_0,\ldots,b_{2k+1}\in B\} \label{eq:P11} \\
\Pc_{22}&=\{b_0a^*b_1ab_2a^*b_3ab_4\cdots a^*b_{2k-1}ab_{2k}a^*b_{2k+1}\mid k\ge0,\,b_0,\ldots,b_{2k+1}\in B\} \label{eq:P22} \\
\Pc_{12}&=\begin{aligned}[t]\bigl\{&b_0ab_1a^*b_2ab_3a^*b_4\cdots ab_{2k-1}a^*b_{2k} \\
 &\quad-\Ec(b_0ab_1a^*b_2ab_3a^*b_4\cdots ab_{2k-1}a^*b_{2k})\mid k\ge1,\,b_0,\ldots,b_{2k}\in B\bigr\} 
 \end{aligned} \label{eq:P12} \\
\Pc_{21}&=\begin{aligned}[t]\bigl\{&b_0a^*b_1ab_2a^*b_3ab_4\cdots a^*b_{2k-1}ab_{2k} \\
 &\quad-\Ec(b_0a^*b_1ab_2a^*b_3ab_4\cdots a^*b_{2k-1}ab_{2k})\mid k\ge1,\,b_0,\ldots,b_{2k}\in B\bigr\}.
 \end{aligned} \label{eq:P21}
\end{align}
Note that $\Pc_{1x}$ means ``starting with $a$'' and $\Pc_{2x}$ means ``starting with $a^*$'', while 
$\Pc_{x1}$ means ``ending with $a$'' and $\Pc_{x2}$ means ``ending with $a^*$.''

When we write that a unitary $u$ {\em normalizes} $B$ in item~\ref{it:u0p} below, we mean that $ubu^*\in B$ and $u^*bu\in B$ for all $b\in B$.

The next result is essentially, a $B$-valued version of Theorem 1.2 of~\cite{NSS01}.
The equivalence of \ref{it:Pmoms}, \ref{it:u0} and \ref{it:cum} was proved in~\cite{SS01};
here we will prove the equivalence of \ref{it:RDiag}--\ref{it:M2}.

\begin{thm}\label{thm:rdiagEq}
Let $a\in A$.
Then the following are equivalent:
\begin{enumerate}[label=(\alph*),leftmargin=20pt]
\item\label{it:RDiag} $a$ is $B$-valued R-diagonal.
\item\label{it:Pmoms}
We have
\begin{equation}\label{eq:Exx0}
\Ec(x_1x_2\cdots x_n)=0
\end{equation}
whenever $n\in\Nats$, $i_0,i_1,\ldots,i_n\in\{1,2\}$ and $x_j\in\Pc_{i_{j-1},i_j}$.

\item\label{it:u0p}
There is a $B$-valued $*$-noncommuative probability space $(\At,\Ect)$, an element $p\in\At$ and a unitary $u\in\At$ such that
\begin{enumerate}[label=(\roman*)]
\item $u$ normalizes $B$,
\item\label{it:pword0thm} $\Ect(p)=0$ and, for all $k\ge1$ and all $b_1,\ldots,b_{2k}\in B$, we have
\[
\Ect(pb_1p^*b_2pb_3p^*b_4\cdots pb_{2k-1}p^*b_{2k}p)=0,
\]
\item $\{u,u^*\}$ is free from $\{p,p^*\}$ with respect to $\Ect$,
\item $\Ect(u)=0$,
\item $a$ and $up$ have the same $B$-valued $*$-distribution.
\end{enumerate}

\item\label{it:u0}
There is an enlargement $(\At,\Ect)$ of $(A,\Ec)$ and a unitary $u\in\At$ such that
\begin{enumerate}[label=(\roman*)]
\item\label{it:u0ucommB} $u$ commutes with every element of $B$,
\item\label{it:u0free} $\{u,u^*\}$ is free from $\{a,a^*\}$ with respect to $\Ect$,
\item\label{it:u0u0} $\Ect(u^k)=0$ for all $k\in\Nats$ (namely, $u$ is Haar unitary),
\item\label{it:u0distr} $a$ and $ua$ have the same $B$-valued $*$-distribution.
\end{enumerate}

\item\label{it:u}
If $(\At,\Ect)$ is an enlargement of $(A,\Ec)$ and $u\in\At$ is a  unitary such that
\begin{enumerate}[label=(\roman*)]
\item\label{it:ucommB} $u$ commutes with every element of $B$,
\item\label{it:ufreefroma} $\{u,u^*\}$ is free from $\{a,a^*\}$ with respect to $\Ect$,
\end{enumerate}
then $a$ and $ua$ have the same $B$-valued $*$-distribution.

\item\label{it:M2}
Consider the subalgebra
\[
B^{(2)}:=\left\{\left(\begin{matrix} b^{(1)}&0\\0&b^{(2)}\end{matrix}\right)\;\bigg|\;b^{(1)},b^{(2)}\in B\right\}
\subseteq M_2(A),
\]
the conditional expectation $\Ec^{(2)}:M_2(A)\to B^{(2)}$ given by
\[
\Ec^{(2)}\left(\left(\begin{matrix} a_{11}&a_{12}\\a_{21}&a_{22}\end{matrix}\right)\right)
=\left(\begin{matrix} \Ec(a_{11})&0\\0&\Ec(a_{22})\end{matrix}\right),
\]
the subalgebra $M_2(B)\subseteq M_2(A)$ and the operator
\[
z=\left(\begin{matrix} 0&a\\a^*&0\end{matrix}\right).
\]
Then $\{z\}$ and $M_2(B)$ are free with respect to $\Ec^{(2)}$ (i.e., with amalgamation over $B^{(2)}$).

\item\label{it:cum}
Letting $a_1=a$ and $a_2=a^*$, for every $n\in\Nats$ and $j=(j(1),\ldots,j(n))\in\{1,2\}^n$,
the $B$-valued cumulant map $\alpha_j$ for the pair $(a_1,a_2)$ is equal to zero if $j$ is either not alternating or is not of even length,
i.e., if $j$ is not of the form $(1,2,1,2,\ldots,1,2)$ or $(2,1,2,1,\ldots,2,1)$.
\end{enumerate}
\end{thm}

The rest of the section is devoted to the proof of \ref{it:RDiag}--\ref{it:M2}.

\begin{proof}[Proof of \ref{it:RDiag} $\Longleftrightarrow$ \ref{it:Pmoms}.]
This is clear, because every expectation on the left-hand-side of~\eqref{eq:IPcond} in Definition~\ref{def:rdiag}
is one of the expectations on the left-hand-side of~\eqref{eq:Exx0}, and vice-versa
and because the vanishing of odd, alternating moments of $a$ is equivalent to $\Pc_{11}\subseteq\ker\Ec$.
\end{proof}

Next is a straighforward computation that we will use twice, so we state it as a separate lemma.
\begin{lemma}\label{lem:uAfr}
Let $(A,\Ec)$ be a $B$-valued $*$-noncommutative probability space.
Assume $u,p\in A$ are such that
\begin{enumerate}[label=(\roman*),leftmargin=30pt]
\item\label{it:unormalizes} $u$ a unitary that normalizes $B$
\item $\{u,u^*\}$ and $\{p,p^*\}$ are free with respect to $\Ec$.
\end{enumerate}
Let $a=up$ and let $\Pc_{ij}$ be as in~\eqref{eq:P11}--\eqref{eq:P21}.
Let $\Afr=\alg(B\cup\{p,p^*\})$ and $\Afr\oup=\Afr\cap\ker \Ec$.
Then
\begin{align}
\Pc_{12}&\subseteq u\Afr\oup u^*, \label{eq:P12A} \\
\Pc_{21}&\subseteq\Afr\oup, \label{eq:P21A} \\
\Pc_{11}&\subseteq u\Afr, \label{eq:P11A} \\
\Pc_{22}&\subseteq \Afr u^*. \label{eq:P22A}
\end{align}
\end{lemma}
\begin{proof}
For $b\in B$, we use the notational convention $b'=u^*bu\in B$.
We examine $\Pc_{12}$.
Consider, for $b_0,\ldots,b_{2k}\in B$,
\[
y:=b_0ab_1a^*b_2ab_3a^*\cdots b_{2k-2}ab_{2k-1}a^*b_{2k}
=b_0upb_1p^*b_2'pb_3p^*\cdots b_{2k-2}'pb_{2k-1}p^*b_{2k}'u^*.
\]
Letting $\yt=pb_1p^*b_2'pb_3p^*\cdots b_{2k-2}'pb_{2k-1}p^*b_{2k}'$, using freeness and condition~\ref{it:unormalizes},
we have
\[
\Ec(y)=b_0\Ec(u\yt u^*)=b_0u\Ec(\yt)u^*.
\]
so an arbitrary element of $\Pc_{12}$ can be written in the form
\[
x=y-\Ec(y)=b_0u(\yt-\Ec(\yt))u^*=u(b_0'(\yt-\Ec(\yt)))u^*.
\]
We have shown~\eqref{eq:P12A}.

Proving~\eqref{eq:P21A} is even easier.
Indeed, we have
\[
z:=b_0a^*b_1ab_2a^*b_3a\cdots b_{2k-2}a^*b_{2k-1}ab_{2k}
=b_0p^*b_1'pb_2p^*b_3'p\cdots b_{2k-2}p^*b_{2k-1}'pb_{2k}\in\Afr,
\]
so an arbitrary element of $\Pc_{21}$ can be written
$x=z-\Ec(z)\in\Afr\oup$.
Similarly, an arbitrary element of $\Pc_{11}$ is of the form
\begin{equation}\label{eq:P11elt}
b_0ab_1a^*b_2ab_3\cdots a^*b_{2k}ab_{2k+1}=ub_0'pb_1p^*b_2'pb_3\cdots p^*b_{2k}'pb_{2k+1}\in u\Afr,
\end{equation}
proving~\eqref{eq:P11A}.
Taking conjugates proves~\eqref{eq:P22A}.
\end{proof}

The next lemma is an analogue of Proposition~2.3 of \cite{NSS01}.
\begin{lemma}\label{lem:2.3}
Let $(A,\Ec)$ be a $B$-valued $*$-noncommutative probability space.
Assume $u,p\in A$ satisfy
\begin{enumerate}[label=(\roman*),labelwidth=3ex,leftmargin=30pt]
\item $u$ is a unitary that normalizes $B$,
\item $\{u,u^*\}$ is free from $\{p,p^*\}$ with respect to $\Ec$,
\item $\Ec(u)=0$,
\item\label{it:pword0} $\Ec(p)=0$ and, for all $k\ge1$ and all $b_1,\ldots,b_{2k}\in B$, we have
\[
\Ec(pb_1p^*b_2pb_3p^*b_4\cdots pb_{2k-1}p^*b_{2k}p)=0.
\]
\end{enumerate}
Then the element $a=up$ satisfies the condition~\ref{it:Pmoms} of Theorem~\ref{thm:rdiagEq}.
\end{lemma}
\begin{proof}
From Lemma~\ref{lem:uAfr} we have~\eqref{eq:P12A} and~\eqref{eq:P21A}.
Examining~\eqref{eq:P11elt} in the proof of Lemma~\ref{lem:uAfr} and invoking condition~\ref{it:pword0},
(then also taking conjugates), we find
\begin{align}
\Pc_{11}&\subseteq u\Afr\oup, \label{eq:P11Ao} \\
\Pc_{22}&\subseteq \Afr\oup u^*. \label{eq:P22Ao}
\end{align}
Using \eqref{eq:P12A}--\eqref{eq:P21A} and \eqref{eq:P11Ao}--\eqref{eq:P22Ao}, we see that
any product of the form $x_1x_2\cdots x_n$ with $x_j\in\Pc_{i_j,i_{j+1}}$ for some $i_1,\ldots,i_{n+1}\in\{1,2\}$
can be rewritten as a word with letters belonging to the sets $\Afr\oup$ and $\{u,u^*\}$, in alternating fashion.
By freeness and the hypothesis $\Ec(u)=0$, each such word evaluates to $0$ under $\Ec$.
\end{proof}

Next is an analogue of Lemma~2.5 of \cite{NSS01}.
\begin{lemma}\label{lem:2.5}
Suppose $u\in A$ is a $B$-normalizing Haar unitary element in 
a $B$-valued $*$-noncommutative probabiltiy space $(A,\Ec)$.
Suppose $D\subseteq A$ is a $*$-subalgebra that is free from $\{u,u^*\}$ with respect to $\Ec$.
Let $n\ge1$, $x_1,\ldots,x_n\in D$ and $h_0,h_1,\ldots,h_n\in\Ints$ be such that
\begin{enumerate}[label=(\roman*),labelwidth=3ex,leftmargin=30pt]
\item\label{it:h1} $h_k\ne0$ if $1\le k\le n-1$
\item\label{it:h2} $h_{k-1}h_k\ge0$ and at least one of $h_{k-1}$ and $h_k$ is nonzero whenever $1\le k\le n$ and $\Ec(x_k)\ne0$.
\end{enumerate}
Then
\begin{equation}\label{eq:Eux}
\Ec(u^{h_0}x_1u^{h_1}x_2\cdots u^{h_{n-1}}x_nu^{h_n})=0.
\end{equation}
\end{lemma}
\begin{proof}
The proof is very similar to the proof found in \cite{NSS01}, only slightly different to take $B$ into account.
We use induction on the cardinality $m$, of $\{k\mid\Ec(x_k)\ne0\}$.
If $m=0$, then by freeness,~\eqref{eq:Eux} holds.
If $m>0$, then letting $k$ be least such that $\Ec(x_k)\ne0$, we write
\begin{multline*}
\Ec(u^{h_0}x_1u^{h_1}x_2\cdots u^{h_{n-1}}x_nu^{h_n}) \\
=\Ec(u^{h_0}x_1u^{h_1}x_2\cdots x_{k-1}u^{h_{k-1}}\Ec(x_k)u^{h_k}x_{k+1}u^{h_{k+1}}\cdots x_nu^{h_n}) \\
+\Ec(u^{h_0}x_1u^{h_1}x_2\cdots x_{k-1}u^{h_{k-1}}(x_k-\Ec(x_k))u^{h_k}x_{k+1}u^{h_{k+1}}\cdots x_nu^{h_n}).
\end{multline*}
By the induction hypothesis, the second term on the right-hand-side equals $0$.
Letting
\[
b=u^{h_{k-1}}\Ec(x_k)u^{-h_{k-1}},
\]
we have $b\in B$ and
the first term equals
\begin{equation}\label{eq:EuxIH}
\Ec(u^{h_0}x_1u^{h_1}\cdots x_{k-2}u^{h_{k-2}}x_{k-1}bu^{h_{k-1}+h_k}x_{k+1}u^{h_{k+1}}\cdots x_nu^{h_n}).
\end{equation}
We will show that, by induction hypothesis, the above quantity equals $0$.
Indeed, we have $h_{k-1}h_k\ge0$ and at most one of $h_{k-1}$ and $h_k$ can be zero, so $h_{k-1}+h_k\ne0$.
If $k<n$ and
$\Ec(x_{k+1})\ne0$, then $h_kh_{k+1}\ge0$
and we conclude $(h_{k-1}+h_k)h_{k+1}\ge0$.
If $k>1$ and $\Ec(x_{k-1}b)\ne0$, then $\Ec(x_{k-1})b\ne0$ so $\Ec(x_{k-1})\ne0$.
Consequently, $h_{k-2}h_{k-1}\ge0$.
Thus, $h_{k-2}(h_{k-1}+h_k)\ge0$.
If $k>1$ then we see that the word to which $\Ec$ is applied in~\eqref{eq:EuxIH} satisfies the requirements
\ref{it:h1}-\ref{it:h2} and, applying the induction hypothesis, we conclude that this moment is $0$.
If $k=1$, then the moment~\eqref{eq:EuxIH} becomes
\[
b\Ec(u^{h_{0}+h_1}x_{2}u^{h_{2}}\cdots x_nu^{h_n})
\]
and again, by the induction hypothesis, this is zero.
\end{proof}

The next result is an analogue of Proposition~2.4 of \cite{NSS01}.
\begin{lemma}\label{lem:2.4}
Suppose that $(A,\Ec)$ is a $B$-valued $*$-noncommutative probability space and $u,p\in A$ are such that
\begin{enumerate}[label=(\roman*),labelwidth=3ex,leftmargin=30pt]
\item $u$ is a $B$-normalizing Haar unitary element,
\item $\{u,u^*\}$ is free from $\{p,p^*\}$ with respect to $\Ec$.
\end{enumerate}
Let $a=up$.
Then $a$ satisfies condition~\ref{it:Pmoms} of Theorem~\ref{thm:rdiagEq}.
\end{lemma}
\begin{proof}
Lemma~\ref{lem:uAfr} applies and we may use the inclusions~\eqref{eq:P12A}--\eqref{eq:P22A},
where $\Afr$ is the algebra generated by $B\cup\{p,p^*\}$.
Thus, for arbitrary $n\in\Nats$, $i_1,\ldots,i_{n+1}\in\{1,2\}$ and $x_j\in\Pc_{i_ji_{j+1}}$, the product
$x_1x_2\cdots x_n$ can be re-written in the form $c_0u^{\eps_1}c_1u^{\eps_2}\cdots c_{r-1}u^{\eps_r}c_r$ for some $r\ge0$, some
$c_0,\ldots,c_r\in\Afr$ and some $\eps_1,\ldots,\eps_r\in\{-1,1\}$, where whenever $\Ec(c_j)\ne0$ for some $1\le j\le r-1$,
we have $\eps_j=\eps_{j+1}$ and in the case $r=0$, we have $\Ec(c_0)=0$.
Now Lemma~\ref{lem:2.5} applies and we conclude $\Ec(x_1x_2\cdots x_n)=0$.
\end{proof}

\begin{proof}[Proof of \ref{it:Pmoms} $\Longleftrightarrow$ \ref{it:u0p}
$\Longleftrightarrow$ \ref{it:u0} $\Longleftrightarrow$ \ref{it:u} of Theorem~\ref{thm:rdiagEq}.]
We first show \ref{it:u}$\implies$\ref{it:u0}.
There is a $B$-valued $*$-noncommutative probability space $(\Bt,\Fc)$ with a unitary $v\in\Bt$ such that
$v$ commutes with every element of $B$ and $\Fc(v^n)=0$ for all $n\in\Nats$.
For example, we could take $\Bt$ to be the algebra $B\otimes\Cpx[C_\infty]$, where $\Cpx[C_\infty]$ is the group $*$-algebra
of the cyclic group of infinite order and where $\Fc(b\otimes x)=b\,x_e$,
where, for $x\in\Cpx[\Ints]$, $x_e$ equals the coefficient in $x$ of the identity element $e\in C_\infty$;
let $c$ be a generator of $C_\infty$
and denote also by $c\in\Cpx[C_\infty]$ the corresponding element of the group algebra;
then $v=1\otimes c$ is a unitary with the desired properties.
Then we let
\begin{equation}\label{eq:amalgfp}
(\At,\Ect)=(A,\Ec)*_B(\Bt,\Fc)
\end{equation}
be the algebraic (amalgamated) free product of $B$-valued $*$-noncommutative probability spaces
and let $u\in\At$ be the copy of $v\in\Bt$ arising from the free product construction.
Then $u$ commutes with every element of $B$, $\{u,u^*\}$ and $\{a,a^*\}$ are free and $\Ect(u)=0$.
By~\ref{it:u}, the elements $a$ and $ua$ have the same $B$-valued $*$-distribution.
Therefore, the requirements of~\ref{it:u0} are fulfilled.

We now show \ref{it:u0}$\implies$\ref{it:u0p}, taking $p=a$.
We need only show that condition~\ref{it:pword0thm} of~\ref{it:u0p} holds, namely, that $\Ec(a)=0$ and
\[
\Ec(ab_1a^*b_2ab_3a^*b_4\cdots ab_{2k-1}a^*b_{2k}a)=0.
\]
Since, by hypothesis, $a$ has the same $B$-valued $*$-distribution as $ua$, we have
\[
\Ec(a)=\Ect(ua)=\Ect(u)\Ec(a)=0,
\]
where the second equality is due to freeness namely, condition~\ref{it:u0free} of~\ref{it:u0}, and the last equality
is because $\Ect(u)=0$, namely, condition~\ref{it:u0u0} of~\ref{it:u0}.
Similarly,
since $u$ commutes with every element of $B$,
we have
\begin{multline*}
\Ec(ab_1a^*b_2ab_3a^*b_4\cdots ab_{2k-1}a^*b_{2k}a)
=\Ect(uab_1a^*b_2ab_3a^*b_4\cdots ab_{2k-1}a^*b_{2k}a) \\
=\Ect(u)\Ec(ab_1a^*b_2ab_3a^*b_4\cdots ab_{2k-1}a^*b_{2k}a)=0,
\end{multline*}
with the penultimate equality due to freeness.
Thus, \ref{it:u0p} holds.

The implication \ref{it:u0p}$\implies$\ref{it:Pmoms} follows immediately by an appeal to Lemma~\ref{lem:2.3}.

We now show \ref{it:Pmoms}$\implies$\ref{it:u}.
We suppose $(\At,\Ect)$ is an enlargement of $(A,\Ec)$ and $u\in A$ is a unitary
satisfying~\ref{it:ucommB} and~\ref{it:ufreefroma} of~\ref{it:u}.
We must show that $a$ and $ua$ have the same $B$-valued $*$-distribution.
Taking $(\Bt,\Fc)$ as in the proof of \ref{it:u}$\implies$\ref{it:u0} above and replacing $(\At,\Ect)$ by
the algebraic free product with amalgamation
$(\At,\Ect)*_B(\Bt,\Fc)$,
we may assume there exists a Haar unitary $v\in\At$ such that $v$ commutes with every element of $B$ and
$\{v,v^*\}$ is free from $\{a,u,a^*,u^*\}$.
Thus, the triple
\[
\{v,v^*\},\;\{u,u^*\},\;\{a,a^*\}
\]
forms a free family of sets.

We make the following claims, which we will prove one after the other:
\begin{enumerate}[label=(\Alph*),leftmargin=30pt]
\item\label{cl:a,va} $a$ and $va$ have the same $B$-valued $*$-distribution,
\item\label{cl:ua,uva} $ua$ and $uva$ have the same $B$-valued $*$-distribution,
\item\label{cl:uaCond} each of $uva$ and $ua$ satisfies condition~\ref{it:Pmoms} of Theorem~\ref{thm:rdiagEq},
\item\label{cl:a,ua} $a$ and $ua$ have the same $B$-valued $*$-distribution.
\end{enumerate}

To prove~\ref{cl:a,va},
we note that both $a$ and $va$ satisfy condition~\ref{it:Pmoms} of Theorem~\ref{thm:rdiagEq};
the operator $a$ does so by hypothesis, and the operator $va$ does so by Lemma~\ref{lem:2.4}.
Moreover, for each $k\ge1$ and $b_1,\ldots,b_{2k-1}\in B$, we have
\begin{multline}\label{eq:altmoms1}
\Ect\big((va)b_1(va)^*b_2(va)b_3(va)^*b_4\cdots(va)b_{2k-1}(va)^*\big) \\
=\Ect(vab_1a^*b_2ab_3a^*b_4\cdots ab_{2k-1}a^*v^*)
=\Ect(ab_1a^*b_2ab_3a^*b_4\cdots ab_{2k-1}a^*),
\end{multline}
where the last inequality is due to freeness, and
\begin{equation}\label{eq:altmoms2}
\Ect\big((va)^*b_1(va)b_2(va)^*b_3(va)b_4\cdots (va)^*b_{2k-1}(va)\big)
=\Ect(a^*b_1ab_2a^*b_3ab_4\cdots a^*b_{2k-1}a).
\end{equation}
Thus, by Remark~\ref{rem:l2.1}, $a$ and $va$ have the same $B$-valued $*$-distribution, and the claim is proved.

To prove~\ref{cl:ua,uva},
note that $\{u,u^*\}$ is free from $\{a,a^*\}$ and from $\{va,(va)^*\}$ and, by~\ref{cl:a,va}, $a$ and $va$ are
identically $*$-distributed.
This implies~\ref{cl:ua,uva}.

To prove~\ref{cl:uaCond}, we argue
from Lemma~\ref{lem:2.4}
that  $uva$ satisfies condition~\ref{it:Pmoms} of Theorem~\ref{thm:rdiagEq},
because $uv$ is a $B$-normalizing Haar unitary.
That $ua$ satisfies condition~\ref{it:Pmoms} of Theorem~\ref{thm:rdiagEq}
now follows from~\ref{cl:ua,uva}.

To prove~\ref{cl:a,ua},
note that the analogues of~\eqref{eq:altmoms1} and~\eqref{eq:altmoms2} hold when $v$ is replaced by $u$,
by the same arguments as given above.
Also, both $a$ and $ua$ satisfy condition~\ref{it:Pmoms} of Theorem~\ref{thm:rdiagEq};
the operator $a$ does so by hypothesis, and the operator $ua$ does so by~\ref{cl:uaCond}.
By Remark~\ref{rem:l2.1}, $a$ and $ua$ are identically distributed.
This finishes the proof of \ref{it:Pmoms}$\implies$\ref{it:u}.

This completes the proof of the equivalence of
\ref{it:Pmoms}, \ref{it:u0p}, \ref{it:u0} and \ref{it:u} of Theorem~\ref{thm:rdiagEq}.
\end{proof}

For the next lemma, we use the sets $\Pc_{11}$ and $\Pc_{22}$ as in~\eqref{eq:P11} and~\eqref{eq:P22}, 
and we use additionally the notation
\begin{align*}
\Pct_{12}&=B\cup\bigl\{b_0ab_1a^*b_2ab_3a^*b_4\cdots ab_{2k-1}a^*b_{2k}\mid k\ge1,\,b_0,\ldots,b_{2k}\in B\bigr\} \\  \\
\Pct_{21}&=B\cup\bigl\{b_0a^*b_1ab_2a^*b_3ab_4\cdots a^*b_{2k-1}ab_{2k}\mid k\ge1,\,b_0,\ldots,b_{2k}\in B\bigr\}, 
\end{align*}
so that we have
\begin{equation}\label{eq:Pcts}
\Pc_{12}=\{w-\Ec(w)\mid w\in\Pct_{12}\},\qquad\Pc_{21}=\{w-\Ec(w)\mid w\in\Pct_{21}\}.
\end{equation}
Now we turn to condition~\ref{it:M2} of Theorem~\ref{thm:rdiagEq} and take $z$ and $B^{(2)}$ as defined there.
\begin{lemma}\label{lem:M2alg}
Let $R\subseteq M_2(A)$ be the subalgebra generated by $\{z\}\cup B^{(2)}$.
Then
\begin{equation}\label{eq:R}
R=\lspan\left\{\left(\begin{matrix} r_{11}&r_{12}\\r_{21}&r_{22}\end{matrix}\right)\,\bigg|\,
r_{11}\in\Pct_{12},\,r_{12}\in\Pc_{11},\,r_{21}\in\Pc_{22},\,r_{22}\in\Pct_{21}\right\}.
\end{equation}
\end{lemma}
\begin{proof}
For $b_1^{(j)},b_2^{(j)}\in B$, we have
\begin{gather*}
z\left(\begin{matrix}b_1^{(1)}&0\\0&b_2^{(1)}\end{matrix}\right)=\left(\begin{matrix}0&ab_2^{(1)}\\a^*b_1^{(1)}&0\end{matrix}\right) \\
z\left(\begin{matrix}b_1^{(1)}&0\\0&b_2^{(1)}\end{matrix}\right)z\left(\begin{matrix}b_1^{(2)}&0\\0&b_2^{(2)}\end{matrix}\right)
=\left(\begin{matrix}ab_2^{(1)}a^*b_1^{(2)}&0\\0&a^*b_1^{(1)}ab_2^{(2)}\end{matrix}\right).
\end{gather*}
Using induction on $n$, we easily see that every word of the form
\begin{equation}\label{eq:zword}
\left(\begin{matrix}b_1^{(0)}&0\\0&b_2^{(0)}\end{matrix}\right)z\left(\begin{matrix}b_1^{(1)}&0\\0&b_2^{(1)}\end{matrix}\right)z\left(\begin{matrix}b_1^{(2)}&0\\0&b_2^{(2)}\end{matrix}\right)
\cdots
z\left(\begin{matrix}b_1^{(n)}&0\\0&b_2^{(n)}\end{matrix}\right)
\end{equation}
belongs to the right-hand-side of~\eqref{eq:R},
from which we easily deduce that the inclusion $\subseteq$ in~\eqref{eq:R} holds.
The reverse inclusion is easily proved by judicious choice of $b_1^{(j)}$ and $b_2^{(j)}$ in~\eqref{eq:zword}.
\end{proof}

\begin{proof}[Proof of \ref{it:Pmoms} $\Longleftrightarrow$  \ref{it:M2} of Theorem~\ref{thm:rdiagEq}.]
From Lemma~\ref{lem:M2alg} and~\eqref{eq:Pcts}, we have
\begin{equation}\label{eq:RcapE}
R\cap\ker\Ec^{(2)}=\lspan\left\{\left(\begin{matrix} r_{11}&r_{12}\\r_{21}&r_{22}\end{matrix}\right)\,\bigg|\,
r_{11}\in\Pc_{12},\,r_{12}\in\Pc_{11},\,r_{21}\in\Pc_{22},\,r_{22}\in\Pc_{21}\right\}.
\end{equation}
and, clearly, we have
\[
M_2(B)\cap\ker\Ec^{(2)}=\left\{\left(\begin{matrix} 0&b_1\\b_2&0\end{matrix}\right)\,\bigg|\,b_1,b_2\in B\right\}.
\]
To show \ref{it:Pmoms} $\Longrightarrow$  \ref{it:M2}, take $r^{(1)},\ldots,r^{(n)}\in R\cap\ker\Ec^{(2)}$
and
\[
b^{(j)}=\left(\begin{matrix} 0&b_1^{(j)}\\b_2^{(j)}&0\end{matrix}\right)\in M_2(B)\cap\ker\Ec^{(2)}
\]
for $0\le j\le n$
and consider the product $r^{(1)}b^{(1)}\cdots r^{(n)}b^{(n)}$
Since each $r^{(j)}b^{(j)}$ is a $2\times2$ matrix whose $(k,l)$-th entry belonging to $P_{kl}$, for each $k,l\in\{1,2\}$,
we see that every entry of the $2\times2$ matrix
$r^{(1)}b^{(1)}\cdots r^{(n)}b^{(n)}$ is a sum of products of the form of the form that condition~\ref{it:Pmoms}
of Theorem~\ref{thm:rdiagEq} guarantees has expectation zero.
Thus, every entry of the $2\times2$ matrix
$r^{(1)}b^{(1)}\cdots r^{(n)}b^{(n)}$ evaluates to zero under $\Ec$.
The same remains true if we left multiply by $b^{(0)}$
or if we choose $b_1^{(n)}=b_2^{(n)}=1$ and then right multiply by $\left(\begin{smallmatrix} 0&1\\1&0\end{smallmatrix}\right)$,
or if we do both.
This shows that each of the products
\begin{multline}\label{eq:rbprods}
r^{(1)}b^{(1)}\cdots r^{(n)}b^{(n)},\quad r^{(1)}b^{(1)}r^{(2)}\cdots b^{(n-1)}r^{(n)}, \\
b^{(0)}r^{(1)}b^{(1)}\cdots r^{(n)}b^{(n)},\quad b^{(0)}r^{(1)}\cdots b^{(n-1)}r^{(n)},
\end{multline}
expects to zero under $\Ec^{(2)}$.
This is \ref{it:M2} of Theorem~\ref{thm:rdiagEq}.

To show \ref{it:M2} $\Longrightarrow$ \ref{it:Pmoms}, using~\eqref{eq:RcapE}, it is straightforward to arrange, given any
$i_1,\ldots,\linebreak[1] i_{n+1}\in\{1,2\}$
and any $x_j\in\Pc_{i_ji_{j+1}}$, that the product $x_1x_2\cdots x_n$ arise as either the $(1,1)$ or $(2,2)$ entry of a product of one of the forms~\eqref{eq:rbprods}
for suitable $r^{(1)},\ldots,r^{(n)}$ and for each $b^{(j)}=\left(\begin{smallmatrix} 0&1\\1&0\end{smallmatrix}\right)$.
Thus, $\Ec(x_1x_2\cdots x_n)=0$.
\end{proof}

Since the equivalence \ref{it:cum} $\Longleftrightarrow$ \ref{it:Pmoms} from Theorem~\ref{thm:rdiagEq} was proved in~\cite{SS01},
the proof of Theorem~\ref{thm:rdiagEq} is complete.

\section{Formal power series}
\label{sec:series}

In this section, we derive some formal power series relations involving moments and cumulants of R-diagonal elements.
The main result is essentially a variation on the combinatorial proof of the power series relation for the R-transform,
here modified to handle R-diagonal elements.
The result is of intrinsic interest and will be used in the appendix to investigate the particular operator considered in Example~\ref{ex:nofreepolar}.

Throughout this section, $B$ will be a unital $*$-algebra, $(A,\Ec)$ will be a $B$-valued $*$-noncommutative probability space
and $a\in A$ will be a $B$-valued R-diagonal element.
We let $a_1=a$ and $a_2=a^*$ and let $\alpha_j$ for sequences $j$ in $\{1,2\}$ denote the $B$-valued cumulant maps of the pair $(a_1,a_2)$.
For $k\in\Nats$, for the $2k-1$-multilinear cumulant maps that are the only ones that need not be zero,
we write
\[
\alpha_k^{(1)}=\alpha_{(\underset{2k}{\underbrace{\scriptstyle 1,2,1,2\ldots,1,2}})},
\qquad\alpha_k^{(2)}=\alpha_{(\underset{2k}{\underbrace{\scriptstyle 2,1,2,1\ldots,2,1}})}.
\]

\begin{prop}\label{prop:mrecurse}
For $k\in\Nats$ and $b_1,\ldots,b_{2k}\in B$, let us write
\begin{align*}
m_k^{(1)}(b_1,\ldots,b_{2k})&=\Ec(ab_1a^*b_2ab_3a^*b_4\cdots ab_{2k-1}a^*b_{2k}) \\
m_k^{(2)}(b_1,\ldots,b_{2k})&=\Ec(a^*b_1ab_2a^*b_3ab_4\cdots a^*b_{2k-1}ab_{2k}),
\end{align*}
and
\[
m_0^{(1)}=m_0^{(2)}=1\in B.
\]
Then, for every $n\ge1$ and $b_1,\ldots,b_{2n}\in B$, we have
\begin{align}
m_n^{(1)}(b_1,\ldots,b_{2n})&=\sum_{\ell=1}^n\sum_{\substack{k(1),\ldots,k(2\ell)\ge0 \\ k(1)+\cdots+k(2\ell)=n-\ell}} \label{eq:m1sum} \\[1ex]
\alpha_\ell^{(1)}\bigg(&
b_{s(1)+1}m_{k(1)}^{(2)}(b_{s(1)+2},\ldots,b_{s(1)+2k(1)+1}), \notag\\
&b_{s(2)+2}m_{k(2)}^{(1)}(b_{s(2)+3},\ldots,b_{s(2)+2k(2)+2}), \displaybreak[2]\notag\\
&b_{s(3)+3}m_{k(3)}^{(2)}(b_{s(3)+4},\ldots,b_{s(3)+2k(3)+3}), \notag\\
&\hspace{7em}\vdots \displaybreak[2]\notag\\
&b_{s(2\ell-2)+2\ell-2}m_{k(2\ell-2)}^{(1)}(b_{s(2\ell-2)+2\ell-1},\ldots,b_{s(2\ell-2)+2k(2\ell-2)+2\ell-2}), \notag\\
&b_{s(2\ell-1)+2\ell-1}m_{k(2\ell-1)}^{(2)}(b_{s(2\ell-1)+2\ell},\ldots,b_{s(2\ell-1)+2k(2\ell-1)+2\ell-1}) \notag\\
\bigg)&b_{s(2\ell)+2\ell}m_{k(2\ell)}^{(1)}(b_{s(2\ell)+2\ell+1},\ldots,b_{s(2\ell)+2k(2\ell)+2\ell}) \notag
\end{align}
and
\begin{align*}
m_n^{(2)}(b_1,\ldots,b_{2n})&=\sum_{\ell=1}^n\sum_{\substack{k(1),\ldots,k(2\ell)\ge0 \\ k(1)+\cdots+k(2\ell)=n-\ell}} \\[1ex]
\alpha_\ell^{(2)}\bigg(&
b_{s(1)+1}m_{k(1)}^{(1)}(b_{s(1)+2},\ldots,b_{s(1)+2k(1)+1}), \\
&b_{s(2)+2}m_{k(2)}^{(2)}(b_{s(2)+3},\ldots,b_{s(2)+2k(2)+2}), \displaybreak[2]\\
&b_{s(3)+3}m_{k(3)}^{(1)}(b_{s(3)+4},\ldots,b_{s(3)+2k(3)+3}), \\
&\hspace{7em}\vdots \displaybreak[2]\\
&b_{s(2\ell-2)+2\ell-2}m_{k(2\ell-2)}^{(2)}(b_{s(2\ell-2)+2\ell-1},\ldots,b_{s(2\ell-2)+2k(2\ell-2)+2\ell-2}), \\
&b_{s(2\ell-1)+2\ell-1}m_{k(2\ell-1)}^{(1)}(b_{s(2\ell-1)+2\ell},\ldots,b_{s(2\ell-1)+2k(2\ell-1)+2\ell-1}) \\
\bigg)&b_{s(2\ell)+2\ell}m_{k(2\ell)}^{(2)}(b_{s(2\ell)+2\ell+1},\ldots,b_{s(2\ell)+2k(2\ell)+2\ell}),
\end{align*}
where $s(j)=2(k(1)+\cdots+k(j-1))$, which includes the case $s(1)=0$.
\end{prop}
\begin{proof}
Using the moment-cumulant formula, we have
\[
m_n^{(1)}(b_1,\ldots,b_{2n})=\sum_{\pi\in\NC(2n)}\alphah_{(1,2,\ldots,1,2)}(\pi)[b_1,\ldots,b_{2n-1}]b_{2n}.
\]
If $\pi\in\NC(2n)$ yields a nonzero term in the above sum, then the block of $\pi$ containing $1$ must contain alternating
odd and even numbers increasing from left to right.
Such a block is indicated in in Figure~\ref{fig:ncp}.
\begin{figure}[ht]
\caption{A noncrossing partition for an alternating $*$-moment.}
\label{fig:ncp}
\begin{picture}(430,70)(0,0)
\put(10,50){\line(1,0){189.5}}
\put(220.5,50){\line(1,0){109.5}}
\put(199.5,50){\line(1,1){7}}
\put(206.5,57){\line(1,-2){7}}
\put(213.5,43){\line(1,1){7}}
\put(10,10){\line(0,1){40}}
\put(90,10){\line(0,1){40}}
\put(170,10){\line(0,1){40}}
\put(250,10){\line(0,1){40}}
\put(330,10){\line(0,1){40}}
\put(5,0){$a$}
\put(85,0){$a^*$}
\put(165,0){$a$}
\put(245,0){$a$}
\put(325,0){$a^*$}
\put(50,25){\oval(60,20)}
\put(38,22){$2k(1)$}
\put(130,25){\oval(60,20)}
\put(118,22){$2k(2)$}
\put(290,25){\oval(60,20)}
\put(265,22){$2k(2\ell-1)$}
\put(380,25){\oval(80,20)}
\put(365,22){$2k(2\ell)$}
\multiput(190,25)(20,0){3}{\circle*{5}}
\end{picture}
\end{figure}
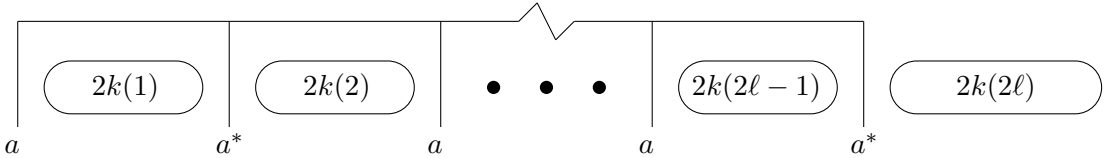
The ovals represent the locations of the other blocks of the partition, and the quantities $2k(j)$ in the ovals are the lengths of the respective gaps,
which may be zero.
If we sum the quantity
\begin{equation}\label{eq:alphah12}
\alphah_{(1,2,\ldots,1,2)}(\pi)[b_1,\ldots,b_{2n-1}]b_{2n}
\end{equation}
over all partitions $\pi\in\NC(2n)$ whose block containing $1$
is the given one shown in the figure, then by the moment-cumulant formula applied to each of the $2\ell$ ovals, we obtain
precisely the summand of the summation appearing in~\eqref{eq:m1sum}.
Now fixing $\ell$ and summing this value over all possible values of $k(1),\ldots,k(2\ell)$ equals the sum of the quantity~\eqref{eq:alphah12}
over all $\pi\in\NC(2n)$ whose first block contains $2\ell$ elements.
Finally, summing over all values of $\ell$ yields the equality~\eqref{eq:m1sum}.

The other equality is proved in the same way, by changing indices.
\end{proof}

\begin{thm}\label{thm:powerseries}
Consider the formal power series
\begin{align}
F(b_1,b_2)&=\sum_{n=0}^\infty\Ec((ab_1a^*b_2)^n) \label{eq:F} \\
G(b_1,b_2)&=\sum_{n=0}^\infty\Ec((a^*b_1ab_2)^n). \label{eq:G}
\end{align}
Then
\begin{align}
F(b_1,b_2)&
\begin{aligned}[t]
=1+\sum_{\ell=1}^\infty\alpha^{(1)}_\ell\Big(b_1G(b_2,b_1),&b_2F(b_1,b_2),b_1G(b_2,b_1), \\
&\ldots,b_2F(b_1,b_2),b_1G(b_2,b_1)\Big)b_2F(b_1,b_2),
\end{aligned} \label{eq:Frecurse} \\
G(b_1,b_2)&
\begin{aligned}[t]
=1+\sum_{\ell=1}^\infty\alpha^{(2)}_\ell\Big(b_1F(b_2,b_1),&b_2G(b_1,b_2),b_1F(b_2,b_1), \\
&\ldots,b_2G(b_1,b_2),b_1F(b_2,b_1)\Big)b_2G(b_1,b_2).
\end{aligned} \label{eq:Grecurse}
\end{align}
\end{thm}
The meaning of the above formulas should be clear.
In~\eqref{eq:F} and~\eqref{eq:G}, the series are formal power series in variables $b_1$ and $b_2$, and for a given $n$,
the corresponding terms should be thought of as being of degree $2n$.
In the formulas~\eqref{eq:Frecurse} and~\eqref{eq:Grecurse}, the formulas on the right-hand-sides
mean the formal power series obtained by substituting and formally expanding.
Since the degree $2n$ term in~\eqref{eq:F} is precisely $m^{(1)}(b_1,b_2,b_1,b_2,\ldots,b_1,b_2)$, and in~\eqref{eq:G}
it is similarly $m^{(2)}(b_1,b_2,b_1,b_2,\ldots,b_1,b_2)$,
the assertions \eqref{eq:Frecurse}--\eqref{eq:Grecurse}  follow from Proposition~\ref{prop:mrecurse}.

The expressions in \eqref{eq:Frecurse}--\eqref{eq:Grecurse} can be treated more precisely and in a more general framework using
the notion of multilinear function series that was described in~\cite{D07}.
Formally, a {\em multilinear function series} over $B$ is a family $X=(\chi_n)_{n\ge0}$ where $\chi_0\in B$ is the constant term and
where, when $n\ge1$,
$\chi_n$ is a $n$-fold multilinear function $\chi_n:B^n\to B$.
The set $\Mul[[B]]$ of all multilinear function series over $B$ was shown in~\cite{D07} to be an associative ring
and a $B$-bimodule under natural operations.
For example, the product is given by $X\Psi$, for $X$ as above and $\Psi=(\psi_n)_{n\ge0}\in\Mul[[B]]$,
where the $0$-th term of $X\Psi$ is $\chi_0\psi_0$ and the $n$-th term, for $n\ge1$, is the multilinear map
\[
(b_1,\ldots,b_n)\mapsto\sum_{k=0}^n\chi_k(b_1,\ldots,b_k)\psi_{n-k}(b_{k+1},\ldots,n).
\]
Moreover, $\Mul[[B]]$ is equipped with a composition operation, (which will be generalized below).
The identity element of $\Mul[[B]]$ with respect to composition is denoted $I$;
it is the multilinear function series whose $n$-th term is $0$ unless $n=1$, in which case it is the identity map on $B$.

\begin{defi}
Let $(b_i)_{i=1}^\infty$ be a sequence in $B$.
Let us say that a {\em variable assignment} for a multlinear function series $X=(\chi_n)_{n=0}^\infty$
is a choice $v:\{(n,j)\in\Nats^2\mid n\ge j\}\to B$.
Then the {\em evaluation} of $X$ at $v$ is just the formal
sum 
\[
X(v)=\chi_0+\sum_{n=1}^\infty\chi_n(v(n,1),\ldots,v(n,n)).
\]
Note that, though this is ostensibly a summation of elements of $B$, no sort of convergence is assumed.
Formally, it may be thought of as the sequence
\[
(\chi_0,\chi_1(v(1,1)),\chi_2(v(2,1),v(2,2)),\ldots)
\]
 in $B$,
and equality of such formal sums amounts to equality of the corresponding sequences.
\end{defi}

For $i\in\{1,2\}$ we let $M^{(i)}\in\Mul[[B]]$ be the multilinear function series
whose odd terms are zero, whose $0$-th term is $1$ and whose $2n$-th term for $n\ge1$
is $m^{(i)}_n$ as described in Proposition~\ref{prop:mrecurse}.
If
\begin{equation}\label{eq:v}
v(n,j)=\begin{cases}
b_1,&j\text{ odd} \\
b_2,&j\text{ even,}
\end{cases}
\end{equation}
then in \eqref{eq:F}--\eqref{eq:G}, $F$  is just $M^{(1)}$ evaluated at $v$ and $G$ is just $M^{(2)}$ evaluated at $v$.

\begin{defi}\label{def:multicomp}
Suppose that $(\Psi^{(i)})_{i\in\Nats}$ is a family in $\Mul[[B]]$,
where each $\Psi^{(i)}=(\psi^{(i)}_n)_{n\ge0}$ has zero constant term and suppose that $X=(\chi_n)_{n\ge0}\in\Mul[[B]]$.
Let $f:\{(n,j)\in\Nats^2\mid n\ge j\}\to\Nats$ be a function.
Then the {\em multivariate composition }
\[
X\overset f\circ(\Psi^{(i)})_{i\in\Nats}
\]
is obtained by, heuristically, replacing the $j$-th argument of $\chi_n$ by $\Psi^{(f(n,j))}$ for each $n$ and $j$.
More precisely,  it
is the element of $\Mul[[B]]$ whose $n$-th term is $\chi_0$ when $n=0$ and when $n\ge1$ it is
the $n$-fold multilinear function that sends $(b_1,\ldots,b_n)$ to
\[
\sum_{p=1}^n\sum_{\substack{k(1),\ldots k(p)\ge1\\k(1)+\cdots+k(p)=n}}
\begin{aligned}[t]
\chi_p\big(
\psi_{k(1)}^{(f(p,1))}(b_1,&\ldots,b_{k(1)}),
\psi_{k(2)}^{(f(p,2))}(b_{k(1)+1},\ldots,b_{k(1)+k(2)}),\cdots, \\
&\psi_{k(p)}^{(f(p,p))}(b_{k(1)+\cdots+k(p-1)+1},\ldots,b_{k(1)+\cdots+k(p-1)+k(p)})
\big ).
\end{aligned}
\]
\end{defi}

The following result is a rephrasing of Proposition~\ref{prop:mrecurse}.
Note that the product $IM^{(i)}$ is given by
$(IM^{(i)})_n=0$ if $n$ is even and
\[
(IM^{(i)})_{2n+1}(b_1,\ldots,b_{2n+1})=\begin{cases}
b_1m^{(i)}_n(b_2,b_3,\ldots,b_{2n+1}),&n\ge1 \\
b_1,&n=0.
\end{cases}
\]
\begin{prop}
For $i=1,2$, let $A^{(i)}$ be the multilinear function series with zero constant term,
whose even terms vanish and whose $(2\ell-1)$-th term is $\alpha^{(i)}_\ell$.
Let $f,g:\{(n,j)\in\Nats^2\mid n\ge j\}\to\{1,2\}$ be given by $f(n,j)\equiv j+1\mod 2$ and $g(n,j)\equiv j\mod 2$.
Then
\begin{align}
M^{(1)}&=1+\big(A^{(1)}\overset f\circ(IM^{(1)},IM^{(2)})\big)M^{(1)} \label{eq:M1recurse} \\
M^{(2)}&=1+\big(A^{(2)}\overset g\circ(IM^{(1)},IM^{(2)})\big)M^{(2)}. \label{eq:M2recurse} 
\end{align}
\end{prop}

Now \eqref{eq:Frecurse}--\eqref{eq:Grecurse} result from evaluating both sides of \eqref{eq:M1recurse}--\eqref{eq:M2recurse} at $v$,
as given in~\eqref{eq:v}.
This, then, is a formal interpretation of Theorem~\ref{thm:powerseries}.

\section{Traces and polar decomposition}
\label{sec:polar}

In this section, we suppose that $B$ is a unital $*$-algebra and $(A,\Ec)$ is a $B$-valued $*$-noncommutative probability space.
We investigate questions of traces and polar decomposition for $B$-valued R-diagonal elements.

First some notation we will use in this section: if $a\in A$, then we write
$a_1=a$ and $a_2=a^*$ and for $k\in\Nats$, let $\beta^{(1)}_k$ and $\beta^{(2)}_k$ be the multilinear maps
\begin{equation}\label{eq:betas}
\beta^{(i)}_k:\underset{2k-1}{\underbrace{B\times\cdots\times B}}\to B
\end{equation}
that are the cumulant maps for the pair $(a_1,a_2)$ corresponding to the alternating sequences
$(1,2,\ldots,1,2)$ and $(2,1,\ldots,2,1)$, respectively, each of length $2k$.
Note that, from Proposition~\ref{prop:alpha*}, we have
\begin{equation}\label{eq:betai*}
\beta^{(i)}_k(b_1,\ldots,b_{2k-1})^*=\beta^{(i)}_k(b_{2k-1}^*,\ldots,b_1^*)
\end{equation}
for all $i=1,2$, $k\ge1$ and $b_1,\ldots,b_{2k-1}\in B$.

The following result is an immediate consequence of Proposition~\ref{prop:trace}.
\begin{prop}\label{prop:traceRdiag}
Suppose $a\in A$ is $B$-valued R-diagonal.
Suppose $\tau$ is a tracial linear functional on $B$.
Then $\tau\circ\Ec$ restricted to $\alg(B\cup\{a,a^*\})$ is tracial if and only if,
for all $k\ge1$ and all $b_1,\ldots,b_{2k}\in B$, we have
\[
\tau(\beta^{(1)}_k(b_1,\ldots,b_{2k-1})b_{2k})=\tau(b_1\beta^{(2)}_k(b_2,\ldots,b_{2k})).
\]
\end{prop}

We now turn to results about $B$-valued R-diagonal elements writen in the form $up$ where $u$ is unitary 
and where $p=p^*$.
First, we make an observation about the polar decompositions of R-diagonal elements in W$^*$-noncommutative probability spaces.
\begin{prop}\label{prop:W*polar}
Suppose $B$ is a von Neumann algebra and $(A,\Ec)$ is a $B$-valued W$^*$-noncommutative probability space.
Suppose $a\in A$ is $B$-valued R-diagonal and $a=v|a|$ is the polar decomposition of $a$.
Then $\Ec(v^n)=0$ for all integers $n\ge1$.
Thus, if $v$ is unitary, then it is a Haar unitary.
\end{prop}
\begin{proof}
Using the amalgamated free product construction for von Neumann algebras,
we get a larger W$^*$-noncommutative probability space
$(\At,\Ect)$ containing a Haar unitary $u$ that commutes with every element of $B$ and such that $\{a,a^*\}$ and $\{u,u^*\}$
are free with respect to $\Ect$.
Thus, $\{|a|,v,v^*\}$ and $\{u,u^*\}$ are free with respect to $\Ect$.
By condition~\ref{it:u} of Theorem~\ref{thm:rdiagEq}, $ua$ has the same $B$-valued $*$-distribution with respect to $\Ect$ as
does $a$ with respect to $\Ec$.
The polar decomposition of $ua$ is $uv|a|$.
Since we are in a W$^*$-noncommutative probability space, given any element $x$ with polar decomposition $x=w|x|$,
the joint $B$-valued $*$-distribution of the pair $(w,|x|)$ is determined by the $B$-valued $*$-distribution of $x$.
Therefore, the $B$-valued $*$-distribution of $uv$ equals that of $v$.
Now, using that $u$ is a Haar unitary that commutes with $B$ and that $u$ and $v$ are $*$-free, it is straightforward to show
$\Ect((uv)^n)=0$ for every integer $n\ge1$.
Thus, also $\Ect(v^n)=0$.
\end{proof}

Proposition~2.6 of \cite{NSS01} shows that in the scalar-valued case, given an R-diagonal element in a tracial
$*$-noncommutative probability space, it has the same $*$-distribution as some element $up$,
where $u$ is a Haar unitary, $p$ is self-adjoint and $u$ is $*$-free from $p$.
The following observation will be used in Example~\ref{ex:nofreepolar} below,
to show that the analogous statement need not hold in the algebra-valued case.
\begin{lemma}\label{lem:notfree}
Suppose in a $B$-valued $*$-noncommutative probability space $(A,\Ec)$,
$a=up\in A$ is $B$-valued $R$-diagonal, where $u\in A$ is unitary,
where $p=p^*\in A$ and where $\{u,u^*\}$ and $\{p\}$ are free
(over $B$) with respect to $\Ec$.
If $\beta_1^{(2)}(1)\in\Cpx1$, then $\beta_1^{(1)}(1)=\beta_1^{(2)}(1)$.
\end{lemma}
\begin{proof}
We have $\beta_1^{(2)}(1)=\Ec(a^*a)=\Ec(p^2)$,
while using freeness of $\{u,u^*\}$ and $\{p\}$, we find
\[
\beta_1^{(1)}(1)=\Ec(up^2u^*)=\Ec(u\Ec(p^2)u^*)=\beta_1^{(2)}(1)\Ec(uu^*)=\beta_1^{(2)}(1)\Ec(1)=\beta_1^{(2)}(1).
\]
\end{proof}

We now study R-diagonal elements that can be written in the form $up$ where $u$ is a $B$-normalizing Haar unitary,
$p=p^*$ and the sets $\{u,u^*\}$ and $\{p\}$ are free over $B$.
It turns out these can be characterized in terms of cumulants.
Our study culminates in Theorem~\ref{thm:C*pos} and Corollary~\ref{cor:polarfree},
which apply to C$^*$- and W$^*$-noncommutative probability spaces, but along the way we prove algebraic versions too.

The following proposition includes an analogue of Proposition 2.6 of \cite{NSS01}.
Note that the equivalence of parts \ref{it:Eaa*}--\ref{it:up} of it are special cases of Proposition~\ref{prop:HUtheta}, below.
However, we include the statment and proof of it, because both are easier to follow than the more general case,
and serve as templates.
\begin{prop}\label{prop:HU}
Let $a\in A$.
Then the following are equivalent:
\begin{enumerate}[label=(\alph*),leftmargin=20pt]
\item\label{it:Eaa*} $a$ is $B$-valued R-diagonal and for all $k\ge1$ and all $b_1,\ldots,b_{2k-1}\in B$, we have
\begin{equation}\label{eq:Eaba*}
\Ec(a^*b_1ab_2a^*b_3a\cdots b_{2k-2}a^*b_{2k-1}a)=\Ec(ab_1a^*b_2ab_3a^*\cdots b_{2k-2}ab_{2k-1}a^*).
\end{equation}
\item\label{it:beta12} $a$ is $B$-valued R-diagonal and for all $k\ge1$, we have agreement $\beta^{(1)}_k=\beta^{(2)}_k$ of the $2k$-th order
$B$-valued cumulant maps
\item\label{it:up0} There exists a $B$-valued $*$-noncommutative probability space $(\At,\Ect)$
with elements $p,u\in\At$ such that 
\begin{enumerate}[label=(\roman*)]
\item\label{it:psa} $p=p^*$
\item\label{it:uHU} $u$ is a Haar unitary that commutes with every element of $B$
\item\label{it:upfree} $\{u,u^*\}$ and $\{p\}$ are free (over $B$) with respect to $\Ec$
\item\label{it:updista} the element $up$ has the same $B$-valued $*$-distribution as $a$
\item\label{it:podds} the odd moments of $p$ vanish, i.e., for every $k\ge0$ and all $b_1,\ldots,b_{2k}\in B$, we have
\begin{equation}\label{eq:poddmoms}
\Ect(pb_1pb_2\cdots pb_{2k}p)=0.
\end{equation}
\end{enumerate}
\item\label{it:up} There exists a $B$-valued $*$-noncommutative probability space $(\At,\Ect)$
with elements $p,u\in\At$ such that parts \ref{it:psa}--\ref{it:updista} of~\ref{it:up0} hold.
\item\label{it:adista*} $a$ is $B$-valued R-diagonal and has the same $B$-valued $*$-distribution as $a^*$.
\end{enumerate}
In part~\ref{it:up0} or~\ref{it:up}, the even moments of $p$ are given by,
for every $k\ge1$ and all $b_1,\ldots,b_{2k-1}\in B$,
\begin{equation}\label{eq:pevenmoms}
\Ect(pb_1pb_2\cdots pb_{2k-1}p)=\Ec(a^*b_1ab_2a^*b_3a\cdots b_{2k-2}a^*b_{2k-1}a).
\end{equation}
\end{prop}
\begin{proof}
The equivalence of~\ref{it:adista*}, \ref{it:Eaa*} and~\ref{it:beta12} is clear from the moment-cumulant formula and Remark~\ref{rem:l2.1}.
Indeed, using Remark~\ref{rem:a*Rdiag}, condition~\ref{it:beta12} says that $a$ and $a^*$ have the same $B$-valued cumulants,
since all other cumulants of $B$-valued R-diagonal elements must vanish.

We now show \ref{it:beta12}$\implies$\ref{it:up0}.
Endow $B\langle X\rangle$ with the $*$-operation that extends the given one on $B$ and satisfies $X^*=X$.
Let $\Theta:B\langle X\rangle\to B$ be the $B$-valued distribution of an element whose odd order cumulant maps
are all zero and whose $B$-valued cumulant map of order $2k$ is $\beta^{(1)}_k$, for all $k\ge1$;
indeed, $\Theta$ is constructed using the moment-cumulant formula.
Write $p$ for the element $X$ of $B\langle X\rangle$.
By Proposition~\ref{prop:alpha*} and the identity~\eqref{eq:betai*}, $\Theta$ is a self-adjoint map.
By the amalgamated free product construction, as in~\eqref{eq:amalgfp},
there is a $B$-valued $*$-noncommutative probability space $(\At,\Ect$)
containing elements $p,u\in\At$ satisfying conditions~\ref{it:psa}, \ref{it:uHU} and~\ref{it:upfree} of~\ref{it:up0}.
The moment-cumulant formula and the fact that $\beta^{(1)}_k=\beta^{(2)}_k$ for all $k$
imply that the identities~\eqref{eq:poddmoms} and~\eqref{eq:pevenmoms} hold;
in particular, condition~\ref{it:podds} of~\ref{it:up0} holds.
Let $\at=up\in\At$.
By Lemma~\ref{lem:2.3}, $\at$ is $B$-valued R-diagonal.
In order to prove that $a$ and $\at$ have the same $B$-valued $*$-distribution, using Remark~\ref{rem:l2.1}, it
will suffice to show that we always have
\begin{align}
\Ec(a^*b_1ab_2a^*b_3ab_4\cdots a^*b_{2k-1}a)&=\Ect(\at^*b_1\at b_2\at^*b_3\at b_4\cdots\at^*b_{2k-1}\at) \label{eq:balmoms1} \\
\Ec(ab_1a^*b_2ab_3a^*b_4\cdots ab_{2k-1}a^*)&=\Ect(\at b_1\at^*b_2\at b_3\at^*b_4\cdots\at b_{2k-1}\at^*). \label{eq:balmoms2}
\end{align}
To show~\eqref{eq:balmoms1}, we use
\begin{equation}\label{eq:at*bat}
\at^*b_1\at b_2\at^*b_3\at b_4\cdots\at^*b_{2k-1}\at=pb_1pb_2\cdots pb_{2k-1}p,
\end{equation}
and~\eqref{eq:pevenmoms}.
To show~\eqref{eq:balmoms2}, we use
\[
\at b_1\at^*b_2\at b_3\at^*b_4\cdots\at b_{2k-1}\at^*=u(pb_1pb_2\cdots pb_{2k-1}p)u^*.
\]
By freeness and the fact that $u$ commutes with every element of $B$, we get
\begin{equation}\label{eq:atbat*}
\Ect(\at b_1\at^*b_2\at b_3\at^*b_4\cdots\at b_{2k-1}\at^*)=\Ect(pb_1pb_2\cdots pb_{2k-1}p)
\end{equation}
so, by appeal to~\eqref{eq:pevenmoms} and the fact that $a$ and $a^*$ have the same $B$-valued $*$-distribution, we get~\eqref{eq:balmoms2}.
Thus, $a$ and $\at=up$ have the same $B$-valued $*$-distribution; namely,~\ref{it:up0} holds.

Clearly, \ref{it:up0}$\implies$\ref{it:up}.

Now we will show \ref{it:up} implies \ref{it:Eaa*} and~\eqref{eq:pevenmoms}.
Let $\at=up$.
That $\at$ is $B$-valued R-diagonal follows from Lemma~\ref{lem:2.4}.
But~\eqref{eq:at*bat} and~\eqref{eq:atbat*} hold by the same arguments as above.
Thus, we have~\eqref{eq:Eaba*}.
\end{proof}

Here is a more general version of Proposition~\ref{prop:HU}.

\begin{prop}\label{prop:HUtheta}
Let $a\in A$ and let $\theta$ be an automorphism of the $*$-algebra $B$.
Then the following are equivalent:
\begin{enumerate}[label=(\alph*),leftmargin=20pt]
\item\label{it:Eaa*theta} $a$ is $B$-valued R-diagonal and for all $k\ge1$ and all $b_1,\ldots,b_{2k-1}\in B$, we have
\begin{multline}\label{eq:Eaba*theta}
\Ec(a^*b_1a\theta(b_2)a^*b_3a\cdots\theta(b_{2k-2})a^*b_{2k-1}a) \\
=\theta\big(\Ec(a\theta(b_1)a^*b_2a\theta(b_3)a^*\cdots b_{2k-2}a\theta(b_{2k-1})a^*)\big).
\end{multline}
\item\label{it:beta12theta}  $a$ is $B$-valued R-diagonal and for all $k\ge1$, the $2k$-th order
$B$-valued cumulant maps satisfy, for all $b_1,\ldots,b_{2k-1}\in B$,
\begin{equation}\label{eq:beta12theta}
\beta^{(2)}_k(b_1,\theta(b_2),b_3,\ldots,\theta(b_{2k-2}),b_{2k-1}) \\
=\theta\big(\beta^{(1)}_k(\theta(b_1),b_2,\theta(b_3),\ldots,b_{2k-2},\theta(b_{2k-1}))\big).
\end{equation}
\item\label{it:up0theta} There exists a $B$-valued $*$-noncommutative probability space $(\At,\Ect)$
with elements $p,u\in\At$ such that 
\begin{enumerate}[label=(\roman*)]
\item\label{it:psatheta} $p=p^*$
\item\label{it:uHUtheta} $u$ is a Haar unitary that normalizes $B$ and satisfies $u^*bu=\theta(b)$ for all $b\in B$
\item\label{it:upfreetheta} $\{u,u^*\}$ and $\{p\}$ are free (over $B$) with respect to $\Ect$
\item\label{it:updistatheta} the element $up$ has the same $B$-valued $*$-distribution as $a$
\item\label{it:poddstheta} the odd moments of $p$ vanish, i.e., for every $k\ge0$ and all $b_1,\ldots,b_{2k}\in B$, we have
\begin{equation*}
\Ect(pb_1pb_2\cdots pb_{2k}p)=0.
\end{equation*}
\end{enumerate}
\item\label{it:uptheta} There exists a $B$-valued $*$-noncommutative probability space $(\At,\Ect)$
with elements $p,u\in\At$ such that parts \ref{it:psatheta}--\ref{it:updistatheta} of~\ref{it:up0theta} hold.
\end{enumerate}
In part~\ref{it:up0theta} or~\ref{it:uptheta},
the even moments of $p$ are given by,
for every $k\ge1$ and all $b_1,\ldots,b_{2k-1}\in B$,
\begin{equation}\label{eq:pevenmomstheta}
\Ect(pb_1pb_2\cdots pb_{2k-1}p)=\Ec(a^*\theta^{-1}(b_1)ab_2a^*\theta^{-1}(b_3)a\cdots b_{2k-2}a^*\theta^{-1}(b_{2k-1})a).
\end{equation}
\end{prop}

In the proof, we will use the following lemma.
\begin{lemma}\label{lem:alphahHU}
Suppose $\theta$ is an automorphism of a $*$-algebra $B$ and $a$ is $B$-valued R-diagonal.
Fix $n\ge1$ and suppose that for all $k\in\{1,\ldots,n-1\}$ and all $b_1,\ldots,b_{2k-1}$ the equality~\eqref{eq:beta12theta} holds.
Then for all $\pi\in\NC(2n)\backslash\{1_{2n}\}$ and all $b_1,\ldots,b_{2n-1}\in B$, with $\alphah$  from the cumulant maps
for the pair $(a_1,a_2)=(a,a^*)$, we have
\begin{multline}\label{eq:alphah12theta}
\alphah_{(2,1,\ldots,2,1)}(\pi)[b_1,\theta(b_2),b_3,\ldots,\theta(b_{2n-2}),b_{2n-1}] \\
=\theta\big(\alphah_{(1,2,\ldots,1,2)}(\pi)[\theta(b_1),b_2,\theta(b_3),\ldots,b_{2n-2},\theta(b_{2n-1})]\big).
\end{multline}
\end{lemma}
\begin{proof}
This follows by straightforward computation using the recursive formula~\eqref{eq:cumrec} for the maps $\alphah_j(\pi)$
by induction on $|\pi|$, namely,
on the number of blocks in $\pi$.
Suppose $|\pi|=2$.
Then both sides of~\eqref{eq:alphah12theta} vanish unless both blocks are of even length.
If both blocks are intervals, then $\pi=\{\{1,\ldots,2p\},\{2p+1,\ldots,2n\}\}$ for some $1\le p<n$ and using~\eqref{eq:beta12theta}
we have
\begin{align*}
\alphah_{(2,1,\ldots,2,1)}(\pi)&[b_1,\theta(b_2),b_3,\ldots,\theta(b_{2n-2}),b_{2n-1}] \\
&=\begin{aligned}[t]
 &\alpha_{(2,1,\ldots,2,1)}(b_1,\theta(b_2),b_3,\ldots,\theta(b_{2p-2}),b_{2p-1})\theta(b_{2p}) \\
 &\quad\cdot\alpha_{(2,1,\ldots,2,1)}(b_{2p+1},\theta(b_{2p+2}),b_{2p+3},\ldots,\theta(b_{2n-2}),b_{2n-1})
\end{aligned} \\
&=\begin{aligned}[t]
 &\theta\Big(\alpha_{(1,2,\ldots,1,2)}(\theta(b_1),b_2,\theta(b_3),\ldots,b_{2p-2},\theta(b_{2p-1}))b_{2p} \\
 &\quad\cdot\alpha_{(1,2,\ldots,1,2)}(\theta(b_{2p+1}),b_{2p+2},\theta(b_{2p+3}),\ldots,b_{2n-2},\theta(b_{2n-1}))\Big)
\end{aligned} \\
&=\theta\big(\alphah_{(1,2,\ldots,1,2)}(\pi)[\theta(b_1),b_2,\theta(b_3),\ldots,b_{2n-2},\theta(b_{2n-1})]\big).
\end{align*}
If $|\pi|=2$ and $\pi$ has an internal interval of the form
$\{2p+1,2p+2,\ldots,2p+2q\}$ for some $1\le p<p+q<n$,
then again using~\eqref{eq:beta12theta} we have
\begin{align}
\alphah_{(2,1,\ldots,2,1)}&(\pi)\big[b_1,\theta(b_2),b_3,\ldots,\theta(b_{2n-2}),b_{2n-1}\big] \label{eq:ii1-1} \displaybreak[2]\\
&=\begin{aligned}[t]
 \alpha_{(2,1,\ldots,2,1)}\Big(&b_1,\theta(b_2),b_3,\ldots,\theta(b_{2p-2}),b_{2p-1}, \\
 &\;\theta(b_{2p})\alpha_{(2,1,\ldots,2,1)}\big(b_{2p+1},\theta(b_{2p+2}),\ldots,b_{2p+2q-1}\big)\theta(b_{2p+2q}), \\
 &\; b_{2p+2q+1},\theta(b_{2p+2p+2}),b_{2p+2q+3},\ldots,\theta(b_{2n-2}),b_{2n-1}\Big)
 \end{aligned} \label{eq:ii1-2} \displaybreak[2] \\
&=\begin{aligned}[t]
 \alpha_{(2,1,\ldots,2,1)}\Big(&b_1,\theta(b_2),b_3,\ldots,\theta(b_{2p-2}),b_{2p-1}, \\
 &\;\theta\Big(b_{2p}\alpha_{(1,2,\ldots,1,2)}\big(\theta(b_{2p+1}),b_{2p+2},\ldots,\theta(b_{2p+2q-1})\big)b_{2p+2q}\Big), \\
 &\; b_{2p+2q+1},\theta(b_{2p+2p+2}),b_{2p+2q+3},\ldots,\theta(b_{2n-2}),b_{2n-1}\Big)
 \end{aligned} \label{eq:ii1-3} \displaybreak[2] \\
&=\begin{aligned}[t]
 \theta\Big(\alpha_{(1,2,\ldots,1,2)}\Big(&\theta(b_1),b_2,\theta(b_3),\ldots,b_{2p-2},\theta(b_{2p-1}), \\
 &\; b_{2p}\alpha_{(1,2,\ldots,1,2)}\big(\theta(b_{2p+1}),b_{2p+2},\ldots,\theta(b_{2p+2q-1})\big)b_{2p+2q}, \\
 &\; \theta(b_{2p+2q+1}),b_{2p+2p+2},\theta(b_{2p+2q+3}),\ldots,b_{2n-2},\theta(b_{2n-1})\Big)\Big)
 \end{aligned} \label{eq:ii1-4} \\
&=\theta\Big(\alphah_{(1,2,\ldots,1,2)}(\pi)\big[\theta(b_1),b_2,\theta(b_3),\ldots,b_{2n-2},\theta(b_{2n-1})\big]\Big). \label{eq:ii1-5}
\end{align}
If $|\pi|=2$ and $\pi$ has an internal interval of the form
$\{2p,2p+1,\ldots,2p+2q-1\}$ for some $1\le p<p+q\le n$,
then again using~\eqref{eq:beta12theta} we have
\begin{align}
&\alphah_{(2,1,\ldots,2,1)}(\pi)\big[b_1,\theta(b_2),b_3,\ldots,\theta(b_{2n-2}),b_{2n-1}\big] \label{eq:ii2-1} \displaybreak[2] \\
&\quad=\begin{aligned}[t]
 \alpha_{(2,1,\ldots,2,1)}\Big(&b_1,\theta(b_2),\ldots,b_{2p-3},\theta(b_{2p-2}), \\
 &\; b_{2p-1}\alpha_{(1,2,\ldots,1,2)}\big(\theta(b_{2p}),b_{2p+1},\ldots,\theta(b_{2p+2q-2})\big)b_{2p+2q-1}, \\
 &\; \theta(b_{2p+2q}),b_{2p+2p+1},\ldots,\theta(b_{2n-2}),b_{2n-1}\Big)
 \end{aligned} \label{eq:ii2-2} \displaybreak[2] \\
&\quad=\begin{aligned}[t]
 \alpha_{(2,1,\ldots,2,1)}\Big(&b_1,\theta(b_2),\ldots,b_{2p-3},\theta(b_{2p-2}), \\
 &\; \theta^{-1}\Big(\theta(b_{2p-1})\alpha_{(2,1,\ldots,2,1)}\big(b_{2p},\theta(b_{2p+1}),\ldots,b_{2p+2q-2}\big)\theta(b_{2p+2q-1})\Big), \\
 &\; \theta(b_{2p+2q}),b_{2p+2p+1},\ldots,\theta(b_{2n-2}),b_{2n-1}\Big)
 \end{aligned} \label{eq:ii2-3} \displaybreak[2] \\
&\quad=\begin{aligned}[t]
 \theta\Big(\alpha_{(1,2,\ldots,1,2)}\Big(&\theta(b_1),b_2,\ldots,\theta(b_{2p-3}),b_{2p-2}, \\
 &\; \theta(b_{2p-1})\alpha_{(2,1,\ldots,2,1)}\big(b_{2p},\theta(b_{2p+1}),\ldots,b_{2p+2q-2}\big)\theta(b_{2p+2q-1}), \\
 &\; b_{2p+2q},\theta(b_{2p+2p+1}),\ldots,b_{2n-2},\theta(b_{2n-1})\Big)\Big)
 \end{aligned} \label{eq:ii2-4} \\
&\quad=\theta\Big(\alphah_{(1,2,\ldots,1,2)}(\pi)\big[\theta(b_1),b_2,\theta(b_3),\ldots,b_{2n-2},\theta(b_{2n-1})\big]\Big). \label{eq:ii2-5}
\end{align}
This finishes the proof of the case $|\pi|=2$.

The induction step when $|\pi|>2$ is very similar.
We see that both sides of~\eqref{eq:alphah12theta} vanish unless $\pi$ has an internal interval of even length,
and then, using the induction hypothesis, one gets recursive formulas like in \eqref{eq:ii1-1}--\eqref{eq:ii2-5},
except that in \eqref{eq:ii1-2}-\eqref{eq:ii1-4} and \eqref{eq:ii2-2}-\eqref{eq:ii2-4}, each $\alpha_{(1,2,\ldots,1,2)}(\cdots)$
appearing immediately after an equality sign
is replaced by $\alpha_{(1,2,\ldots,1,2)}(\pi')[\cdots]$ and each $\alpha_{(2,1,\ldots,2,1)}(\cdots)$
is replaced by $\alpha_{(2,1,\ldots,2,1)}(\pi')[\cdots]$,
where $\pi'$ is obtained from $\pi$ by removing the corresponding internal interval and renumbering.
\end{proof}

\begin{proof}[Proof of Proposition~\ref{prop:HUtheta}.]
The proof is patterned after the proof of Proposition~\ref{prop:HU}.
The equivalence \ref{it:Eaa*theta}$\Longleftrightarrow$\ref{it:beta12theta} follows easily from
the moment-cumulant formula~\eqref{eq:alphacummom}
and Lemma \ref{lem:alphahHU}, by induction on $k$, keeping in mind that all of the cumulant maps vanish
except  the $2k$-th order ones $\beta^{(1)}_k=\alpha_{(1,2,\ldots,1,2)}$ and $\beta^{(2)}_k=\alpha_{(2,1,\ldots,2,1)}$, ($k\ge1$).

We now show \ref{it:beta12theta}$\implies$\ref{it:up0theta}.
Endow $B\langle X\rangle$ with the $*$-operation that extends the given one on $B$ and satisfies $X^*=X$.
Let $\Theta:B\langle X\rangle\to B$ be the $B$-valued distribution of an element whose odd order $B$-valued cumulant maps
are all zero and whose $B$-valued cumulant map of order $2k$ is the map $\gamma_k$ given by
\begin{multline}\label{eq:gammakbeta}
\gamma_k(b_1,\ldots,b_{2k-1})=\beta^{(2)}_k\big(\theta^{-1}(b_1),b_2,\theta^{-1}(b_3),\ldots,b_{2k-2},\theta^{-1}(b_{2k-1})\big) \\
=\theta\circ\beta^{(1)}_k\big(b_1,\theta^{-1}(b_2),b_3,\ldots,\theta^{-1}(b_{2k-2}),b_{2k-1}),
\end{multline}
for all $k\ge1$, where the second equality above is from~\eqref{eq:beta12theta}.
Write $p$ for the element $X$ of $B\langle X\rangle$.
By Proposition~\ref{prop:alpha*} and the identity~\eqref{eq:betai*}, $\Theta$ is a self-adjoint map.
We claim (letting $\gamma_k$ be shorthand for $\gamma_{(1,1,\ldots,1)}$ with $1$ repeated $2k$ times
and similarly letting $\gammah_k$ be shorthand for $\gammah_{(1,1,\ldots,1)}$), that for every $k\ge1$,
every $\pi\in\NC(2k)$ and every $b_1,\ldots,b_{2k-1}\in B$, we have
\begin{equation}\label{eq:gammaalpha}
\gammah_k(\pi)[b_1,\ldots,b_{2k-1}]
=\alphah_{(2,1,\ldots,2,1)}(\pi)\big[\theta^{-1}(b_1),b_2,\theta^{-1}(b_3),\ldots,b_{2k-2},\theta^{-1}(b_{2k-1})\big].
\end{equation}
If $\pi=1_{2k}$ then this holds by definition, and the general case is proved by induction on the number $|\pi|$ of blocks in $\pi$.
For the induction step, if $\pi$ has an internal interval block of the form $V=\{2r,2r+1,\ldots,2r+2s-1\}$, for some integers
$1\le r<r+s\le k$,
then letting $\pi'\in\NC(2k-2s)$ be obtained from $\pi$ by removing the block $V$ and renumbering, we have
\begin{align*}
\gammah_k&(\pi)[b_1,\ldots,b_{2k-1}] \\
&\;=\gammah_{k-s}(\pi')\big[b_1,\ldots,b_{2r-2},b_{2r-1}\gamma_s(b_{2r},\ldots,b_{2r+2s-2})b_{2r+2s-1},b_{2r+2s},\ldots,b_{2k-1}\big] \\
&\;=\begin{aligned}[t]
 \alphah_{(2,1,\ldots,2,1)}(\pi')\Big[&\theta^{-1}(b_1),b_2,\ldots,\theta^{-1}(b_{2r-3}),b_{2r-2}, \\
 &\theta^{-1}\Big(b_{2r-1}\alpha_{(2,1,\ldots,2,1)}\big(\theta^{-1}(b_{2r}),b_{2r+1},\ldots,\theta^{-1}(b_{2r+2s-2})\big)b_{2r+2s-1}\Big), \\
 &b_{2r+2s},\theta^{-1}(b_{2r+2s+1}),\ldots,b_{2k-2},\theta^{-1}(b_{2k-1})\Big]
\end{aligned} \\
&\;=\begin{aligned}[t]
 \alphah_{(2,1,\ldots,2,1)}(\pi')\Big[&\theta^{-1}(b_1),b_2,\ldots,\theta^{-1}(b_{2r-3}),b_{2r-2}, \\
 &\theta^{-1}(b_{2r-1})\alpha_{(1,2,\ldots,1,2)}\big(b_{2r},\theta^{-1}(b_{2r+1}),\ldots,b_{2r+2s-2}\big)\theta^{-1}(b_{2r+2s-1}), \\
 &b_{2r+2s},\theta^{-1}(b_{2r+2s+1}),\ldots,b_{2k-2},\theta^{-1}(b_{2k-1})\Big]
\end{aligned} \\
&\;=\alphah_{(2,1,\ldots,2,1)}(\pi)\big[\theta^{-1}(b_1),b_2,\theta^{-1}(b_3),\ldots,b_{2k-2},\theta^{-1}(b_{2k-1})\big],
\end{align*}
where the first and last equalities are by the recursive formula for $\gammah$ (see~\eqref{eq:cumrec})
the second equality is by the induction hypothesis and~\eqref{eq:gammakbeta}, and the third equality is from~\eqref{eq:beta12theta}.
The cases when $\pi$ has an internal interval block of the form $\{2r+1,\ldots,2r+2s\}$ for some integers $1\le r<r+s<k$
or when $\pi$ has only two blocks, both of them intervals,
are treated similarly, to prove~\eqref{eq:gammaalpha}.
From the equality~\eqref{eq:gammaalpha}, the fact that all odd $B$-valued moments vanish
and the moment-cumulant formula, we deduce the equality~\eqref{eq:pevenmomstheta} for all $k\ge1$ and all
$b_1,\ldots,b_{2k-1}$.

Using a crossed product construction and the amalgamated free product construction,
there is a $B$-valued $*$-noncommutative probability space $(\At,\Ect$)
containing elements $p,u\in\At$ satisfying conditions~\ref{it:psatheta}, \ref{it:uHUtheta} and~\ref{it:upfreetheta} of~\ref{it:up0theta}.
(This should be reasonably clear, but see the proof of Lemma~\ref{lem:C*}, below for more details.)
By construction,~\ref{it:poddstheta} holds.
Let $\at=up\in\At$.
By Lemma~\ref{lem:2.3}, $\at$ is $B$-valued R-diagonal.
In order to prove that $a$ and $\at$ have the same $B$-valued $*$-distribution, using Remark~\ref{rem:l2.1}, it
will suffice to show that the equations~\eqref{eq:balmoms1} and~\eqref{eq:balmoms2} always hold.
To show~\eqref{eq:balmoms1}, we use
\begin{equation}\label{eq:at*battheta}
\at^*b_1\at b_2\at^*b_3\at b_4\cdots\at^*b_{2k-1}\at=p\theta(b_1)pb_2p\theta(b_3)p\cdots\cdots b_{2k-2}p\theta(b_{2k-1})p,
\end{equation}
and~\eqref{eq:pevenmomstheta}.
To show~\eqref{eq:balmoms2}, we use
\begin{equation}\label{eq:atbat*theta}
\at b_1\at^*b_2\at b_3\at^*b_4\cdots\at b_{2k-1}\at^*=u(pb_1p\theta(b_2)pb_3p\cdots \theta(b_{2k-2})pb_{2k-1}p)u^*.
\end{equation}
Thus, we get
\begin{align*}
\Ect(\at b_1\at^*b_2\at b_3\at^*b_4&\cdots\at b_{2k-1}\at^*)=\Ect(u\Ect(pb_1p\theta(b_2)pb_3p\cdots \theta(b_{2k-2})pb_{2k-1}p)u^*) \\
&=\theta^{-1}\big(\Ect(pb_1p\theta(b_2)pb_3p\cdots \theta(b_{2k-2})pb_{2k-1}p)\big) \\
&=\theta^{-1}\big(\Ec(a^*\theta^{-1}(b_1)a\theta(b_2)a^*\theta^{-1}(b_3)a\cdots\theta(b_{2k-2})a^*\theta^{-1}(b_{2k-1})a)\big) \\
&=\Ec(ab_1a^*b_2ab_3a^*\cdots b_{2k-2}ab_{2k-1}a^*)
\end{align*}
where the first equality is due to freeness,  the second equality is by~\ref{it:uHUtheta} of~\ref{it:up0theta},
the third equality is from~\eqref{eq:pevenmomstheta}
and the fourth equality is~\eqref{eq:Eaba*theta}.
This proves we get~\eqref{eq:balmoms2}.
Thus, $a$ and $\at=up$ have the same $*$-distribution; namely,~\ref{it:up0theta} holds.

The implication \ref{it:up0theta}$\implies$\ref{it:uptheta} is trivially true.
Assuming~\ref{it:uptheta}, the equality~\eqref{eq:pevenmomstheta} follows by writing
\begin{multline*}
\Ec(a^*\theta^{-1}(b_1)ab_2a^*\theta^{-1}(b_3)a\cdots b_{2k-2}a^*\theta^{-1}(b_{2k-1})a) \\
=\Ect(pu^*\theta^{-1}(b_1)upb_2pu^*\theta^{-1}(b_3)up\cdots b_{2k-2}pu^*\theta^{-1}(b_{2k-1})up) \\
=\Ect(pb_1pb_2\cdots pb_{2k-1}p).
\end{multline*}

We now show \ref{it:uptheta}$\implies$\ref{it:Eaa*theta}.
That $a$ is $B$-valued R-diagonal follows from Lemma~\ref{lem:2.4}.
To show~\eqref{eq:Eaba*theta}, let $\at=up$.
By hypothesis, $a$ and $\at$ have the same $B$-valued $*$-distribution.
Of course, the equalities~\eqref{eq:at*battheta} and~\eqref{eq:atbat*theta} hold.
Therefore, we have
\begin{align*}
\Ec(a^*b_1a\theta(b_2)a^*b_3a\cdots\theta&(b_{2k-2})a^*b_{2k-1}a) 
=\Ect(\at^*b_1\at\theta(b_2)\at^*b_3\at\cdots\theta(b_{2k-2})\at^*b_{2k-1}\at) \\
&=\Ect\big(p\theta(b_1)p\theta(b_2)p\theta(b_3)p\cdots \theta(b_{2k-2})p\theta(b_{2k-1})p\big)  \displaybreak[2]  \\
&=\theta\big(\Ect\big(u\big(p\theta(b_1)p\theta(b_2)p\theta(b_3)p\cdots \theta(b_{2k-2})p\theta(b_{2k-1})p\big)u^*\big)\big) 
 \displaybreak[2]  \\
&=\theta\big(\Ect\big(\at\theta(b_1)\at^*b_2\at\theta(b_3)p\cdots b_{2k-2}\at\theta(b_{2k-1})\at^*\big)\big) \\
&=\theta\big(\Ec\big(a\theta(b_1)a^*b_2a\theta(b_3)p\cdots b_{2k-2}a\theta(b_{2k-1})a^*\big)\big).
\end{align*}
where the second equality is~\eqref{eq:at*battheta}, the third is by conditions~\ref{it:uHUtheta}
and~\ref{it:upfreetheta} of~\ref{it:up0theta}
and the fourth is~\eqref{eq:atbat*theta}.
\end{proof}

\begin{lemma}\label{lem:C*}
In Proposition~\ref{prop:HUtheta}, if $B$ is a C$^*$-algebra and if
$(A,\Ec)$ is a $B$-valued C$^*$-noncommutative probability space, then
in part~\ref{it:up0theta}, $(\At,\Ect)$ can be realized as a $B$-valued C$^*$-noncommutative probability space.
\end{lemma}
\begin{proof}
In the proof of Proposition~\ref{prop:HUtheta},
the $*$-noncommutative probability space $(\At,\Ect)$ was constructed as an algebraic amalgamated free product
\begin{equation}\label{eq:algamalgfp}
(\At,\Ect)=(B\rtimes_\theta^\alg\Ints,\Fc^\alg)*_B(B\langle X\rangle,\Theta).
\end{equation}
Here $B\rtimes_\theta^\alg\Ints$ is an algebraic crossed product of $B$ by the automorphism $\theta$ and $\Fc^\alg$ is
the canonical conditional expectation, and the unitary $u$ of Proposition~\ref{prop:HUtheta}\ref{it:up0theta}
is from the element $1\in\Ints$;
recall, $\Theta$ is defined abstractly by specifying cumulants, and is the $*$-distribution of the element we called $p$.
If $B$ is a C$^*$-algebra, then we can take, instead, the crossed product C$^*$-algebra 
$B\rtimes_\theta\Ints$ and the canonical conditional expectation $\Fc:B\rtimes_\theta\Ints\to B$.

We also want a $B$-valued C$^*$-noncommutative probability space to replace $(B\langle X\rangle,\Theta)$.
Endow $B\langle X_1,X_2\rangle$ with the $*$-operation determined by that of $B$ and by setting $X_1^*=X_2$
and let $\Thetat:B\langle X_1,X_2\rangle\to B$ denote the $*$-distribution of the pair $(a,a^*)$.
Since $(A,\Ec)$ is a $B$-valued C$^*$-noncommutative probability space,
we have (see, for example, \cite{La95}), the Hilbert $B$-module $E=L^2(A,\Ec)$, on which $A$ acts by left multiplication
as bounded, adjointable operators.
We may without loss of generality assume $A$ is generated as a C$^*$-algebra by $B\cup\{a\}$.
It will be convenient to identify $E$ with $L^2(B\langle X_1,X_2\rangle,\Thetat)$, which is obtained 
by separation and completion, after endowing $B\langle X_1,X_2\rangle$ with the inner product 
$\langle y_1,y_2\rangle=\Thetat(y_1^*y_2)$,
on which $B\langle X_1,X_2\rangle$ acts by left multiplication as bounded, adjointable
operators on $E$.
Let $\pi$ denote this action.
Letting $y\mapsto\yh$ denote the defining mapping $B\langle X_1,X_2\rangle\to E$, the set $\{\bh\mid b\in B\}$
is a complemented subspace of $E$ isomorphic to the Hilbert $B$-module $B$,
with a self-adjoint projection $P:E\to B$.
We have the conditional expectation $\Lc(E)\to B$ given by $Z\mapsto PZ\oneh$,
and we have $\Thetat(y)=P\pi(y)\oneh$ for all $y\in B\langle X_1,X_2\rangle$.

Consider the closed right $B$-submodules of $E$,
\begin{align*}
E_0&=\clspan\{b_0X_2b_1X_1b_2X_2b_3X_1\cdots b_{n-2}X_2b_{n-1}X_1b_n\mid n\text{ even},\;b_0,\ldots,b_n\in B\}, \\
E_1&=\clspan\{b_0X_1b_1X_2b_2X_1b_3X_2\cdots b_{n-2}X_1b_{n-2}X_2b_{n-1}X_1b_n\mid n\text{ odd},\;b_0,\ldots,b_n\in B\}.
\end{align*}
Then $B$ (the image of $P$) is a submodule of $E_0$,
the spaces $E_0$ and $E_1$ are orthogonal to each other
with respect to the $B$-valued inner product
and their sum $E_0+E_1$, which we write $E_0\oplus E_1$,
is also a closed submodule of $E$.
We have the representation $\sigma:B\to\Lc(E_0\oplus E_1)$ given by
\[
\sigma(b)=\pi(b)\restrict_{E_0}\oplus\pi(\theta^{-1}(b))\restrict_{E_1}.
\]
We also have the bounded operator $Y$ on $E_0\oplus E_1$ given by
\[
Y(e_0\oplus e_1)=\pi(X_2)e_1\oplus\pi(X_1)e_0.
\]
Since $\pi(X_1)^*=\pi(X_2)$, we easily see that $Y$ is self-adjoint.
Thus, the representation $\sigma$ of $B$ extends to a representation $\sigma:B\langle X\rangle\to\Lc(E_0\oplus E_1)$ by setting
$\sigma(X)=Y$.
Of course, $\oneh\in B\subset E_0$ and, if we identify $B$ with $\sigma(B)\subset\Lc(E_0\oplus E_1)$, then we have the conditional
expectation $\Gc:\Lc(E_0\oplus E_1)\to B$ given by $\Gc(Z)=PZ\oneh$.
Now we see that, for every $y\in B\langle X\rangle$, we have
$\Gc(\sigma(y))=\Theta(y)$.
Indeed, for a monomial $m=b_0Xb_1X\cdots b_{n-1}Xb_n$, if $n$ is odd, then $m\oneh\in E_1$ and $\Gc(\sigma(m))=0=\Theta(m)$,
while for $n$ even, we have
\[
\sigma(m)\oneh=\pi\big(b_0X_2\theta^{-1}(b_1)X_1b_2X_2\theta^{-1}(b_3)X_1\cdots b_{n-2}X_2\theta^{-1}(b_{n-1})X_1b_n\big)\oneh
\]
and
\begin{align*}
\Gc(\sigma(m))&=\Thetat\big(b_0X_2\theta^{-1}(b_1)X_1b_2X_2\theta^{-1}(b_3)X_1\cdots b_{n-2}X_2\theta^{-1}(b_{n-1})X_1b_n\big) \\
&=b_0\Ec(a^*\theta^{-1}(b_1)ab_2a^*\theta^{-1}(b_3)\cdots b_{n-2}a^*\theta^{-1}(b_{n-1})a)b_n  \\
&=b_0\Ect(pb_1pb_2\cdots pb_{n-1}p)b_n=\Theta(m),
\end{align*}
where for the third equality we used~\eqref{eq:pevenmomstheta}.

Thus, we may replace $(B\langle X\rangle,\Theta)$ with the $B$-valued C$^*$-noncommutative probability space
$(\Lc(E_0\oplus E_1),\Gc)$, and instead of the algebraic amalgamated free product~\eqref{eq:algamalgfp},
we let
\[
(\At,\Ect)=(B\rtimes_\theta\Ints,\Fc)*_B(\Lc(E_0\oplus E_1),\Gc).
\]
Identifying $p$ with $\sigma(X)\in\Lc(E_0\oplus E_1)$ and $u\in B\otimes_\theta\Ints$ with the unitary corresponding to the group
element $1$ in the crossed product construction, we are done.
\end{proof}

\begin{thm}\label{thm:C*pos}
Let $B$ be a C$^*$-algebra and suppose $(A,\Ec)$ is a $B$-valued C$^*$-non\-com\-muta\-tive probability space.
Suppose $a\in A$ and let $\theta$ be a $*$-automorphism of $B$.
Then the following are equivalent:
\begin{enumerate}[label=(\alph*),leftmargin=20pt]
\item\label{it:Eaa*C*pos} $a$ is $B$-valued R-diagonal and for all $k\ge1$ and all $b_1,\ldots,b_{2k-1}\in B$, we have
\begin{multline}\label{eq:Eaba*C*pos}
\Ec(a^*b_1a\theta(b_2)a^*b_3a\cdots\theta(b_{2k-2})a^*b_{2k-1}a) \\
=\theta\big(\Ec(a\theta(b_1)a^*b_2a\theta(b_3)a^*\cdots b_{2k-2}a\theta(b_{2k-1})a^*)\big).
\end{multline}
\item\label{it:beta12C*pos} $a$ is $B$-valued R-diagonal and for all $k\ge1$, the $2k$-th order
$B$-valued cumulant maps $\beta_k^{(1)}$ and $\beta_k^{(2)}$ (notation defined near~\eqref{eq:betas})
satisfy, for all $b_1,\ldots,b_{2k-1}\in B$,
\begin{equation}\label{eq:beta12C*pos}
\beta^{(2)}_k(b_1,\theta(b_2),b_3,\ldots,\theta(b_{2k-2}),b_{2k-1}) \\
=\theta\big(\beta^{(1)}_k(\theta(b_1),b_2,\theta(b_3),\ldots,b_{2k-2},\theta(b_{2k-1}))\big)
\end{equation}
\item\label{it:C*pos} there exists a $B$-valued C$^*$-non\-com\-muta\-tive probability space $(\At,\Ect)$ with elements $u,p\in\At$
satisfying 
\begin{enumerate}[label=(\roman*)]
\item\label{it:uC*pos} $u$ is a Haar unitary that normalizes $B$ and satisfies $u^*bu=\theta(b)$ for all $b\in B$
\item\label{it:pC*pos} $p\ge0$,
\item\label{it:C*posfree} $\{u^*,u\}$ and $\{p\}$ are free (over $B$) with respect to $\Ect$
\item\label{it:C*posdistr} $a$ and $up$ have the same $B$-valued $*$-distribution.
\end{enumerate}
\end{enumerate}
Moreover, the even moments of $p$ in~\ref{it:C*pos} are given by
\begin{equation}\label{eq:pevenmomsthetaC*pos}
\Ect(pb_1pb_2\cdots pb_{2k-1}p)=\Ec(a^*\theta^{-1}(b_1)ab_2a^*\theta^{-1}(b_3)a\cdots b_{2k-2}a^*\theta^{-1}(b_{2k-1})a).
\end{equation}
for all $k\ge1$ and all $b_1,\ldots,b_{2k-1}\in B$.
\end{thm}
\begin{proof}
The equivalence of~\ref{it:Eaa*C*pos} and~\ref{it:beta12C*pos}, as well as the implication \ref{it:C*pos}$\implies$\ref{it:Eaa*C*pos}
follow from Proposition~\ref{prop:HUtheta}.
It remains to show \ref{it:Eaa*C*pos}$\implies$\ref{it:C*pos}.

Assuming~\ref{it:Eaa*C*pos}, by Proposition~\ref{prop:HUtheta} and Lemma~\ref{lem:C*}, there is
a $B$-valued C$^*$-noncommutative probability space $(\At,\Ect)$, 
and elements $u$ and $p=p^*$ of $\At$ such that~\ref{it:uC*pos}, \ref{it:C*posfree} and~\ref{it:C*posdistr} of~\ref{it:C*pos} hold,
the even moments of $p$ are given by~\eqref{eq:pevenmomsthetaC*pos} and the odd moments of $p$ all vanish.
By modifying $(\At,\Ect)$, if necessary, we may assume that $p=s|p|$, where $s$ is a symmetry (i.e., a self-adjoint unitary element)
that commutes with $|p|$ and with every element of $B$, satisfies $\Ect(s)=0$ and such that $\{s,p\}$ is free from $\{u,u^*\}$.
Indeed, we may without loss of generality assume $\At$ is the C$^*$-algebra generated by $\{u,p\}$ and the GNS
representation of $\Ect$ is faithful, in which case
$\At$ is the amalgamated free product of $C^*(\{u\}\cup B)$ and $C^*(\{p\}\cup B)$ over $B$.
We may enlarge $(\At,\Ect)$ to be the amalgamated free product of
$C^*(\{u\}\cup B)$ and $C^*(\{p\}\cup B)\otimes(\Cpx\oplus\Cpx)$, where the conditional expectation of the latter onto $B$ is
$\Ect\restrict_{C^*(\{p\}\cup B)}\otimes\tau$, where $\tau:\Cpx\oplus\Cpx\to\Cpx$ is the state sending $1\oplus-1$ to $0$.
We let
\[
s=1\otimes(1\oplus-1)\in C^*(\{p\}\cup B)\otimes(\Cpx\oplus\Cpx)
\]
and note that the pair $(u,p)$ has the same $B$-valued $*$-distribution as $(u,s|p|)$.
Thus, $a$ has the same $*$-distribution as $us|p|$.
We observe that, with respect to $\Ect$, $us$ is a $B$-normalizing Haar unitary with $(us)^*b(us)=s\theta(b)s=\theta(b)$.

We need only show that $\{us,su^*\}$ is free from $\{|p|\}$ with respect to $\Ect$.
This is straightforward to verify.
Indeed, we need only show that every alternating product in the sets $\{(us)^k\mid k\ge1\}\cup\{(su^*)^k\mid k\ge1\}$
and $\{|p|^k-\Ect(|p|^k)\mid k\ge1\}$ evaluates to zero under $\Ect$.
However, keeping in mind that $s$ commutes with $|p|$ and every element of $B$,
we may rewrite each such product as an alternating product in the sets $\{u,u^*\}$ and
\[
\{|p|^k-\Ect(|p|^k)\mid k\ge1\}\cup\{s\big(|p|^k-\Ect(|p|^k)\big)\mid k\ge1\}\cup\{s\},
\]
and such an alternating product evaluates to $0$ under $\Ect$ by freeness of $\{|p|,s\}$ and $\{u,u^*\}$.
\end{proof}

In the case of $B=\Cpx$, Theorem~\ref{thm:C*pos} reduces to the following, which is certainly well known, though
we didn't find a good reference.
It would follow straighforwardly from Proposition~2.6 of \cite{NSS01}.
\begin{cor}\label{cor:scalarpolar}
Let $(A,\tau)$ be a tracial (scalar-valued) C$^*$-noncommutative probability space and let $a\in A$.
Then the following are equivalent:\
\begin{enumerate}[label=(\alph*),leftmargin=20pt]
\item\label{it:scalarRdiag}  $a$ is R-diagonal.
\item\label{it:scalarup} There exists a tracial (scalar-valued) C$^*$-noncommutative probability space $(\At,\taut)$
and elements $u,p\in\At$ such that\
\begin{enumerate}[label=(\roman*)]
\item\label{it:uscalar} $u$ is a Haar unitary,
\item\label{it:pscalar} $p\ge0$,
\item\label{it:scalarfree} $\{u^*,u\}$ and $\{p\}$ are free with respect to $\taut$
\item\label{it:scalardistr} $a$ and $up$ have the same $*$-distribution.
\end{enumerate}
\end{enumerate}
\end{cor}
\begin{proof}
We apply Theorem~\ref{thm:C*pos} in the case $B=\Cpx$.
Then of course each cumulant map $\beta^{(i)}_k$ is just a real number
and $\theta$ is the identity map.
The condition~\eqref{eq:beta12C*pos} just becomes $\beta^{(2)}_k=\beta^{(1)}_k$ for every $k$, and this holds,
by Proposition~\ref{prop:traceRdiag},
because we assume $\tau$ is a trace.
Now the equivalence of~\ref{it:scalarRdiag} and~\ref{it:scalarup} above follows from the equivalence
of~\ref{it:beta12C*pos} and~\ref{it:C*pos} in Theorem~\ref{thm:C*pos}.
\end{proof}

For an element $a$ in a $B$-valued W$^*$-noncommutative probability space whose polar decomposition is $a=u|a|$, the
joint $B$-valued $*$-distribution of $\{u,|a|\}$ is completely determined by the $B$-valued $*$-distribution of $a$, and vice-versa.
Thus, the following corollary follows from the Theorem~\ref{thm:C*pos}.
\begin{cor}\label{cor:polarfree}
Let $B$ be a von Neumann algebra and suppose $(A,\Ec)$ is a $B$-valued W$^*$-non\-com\-muta\-tive probability space
whose GNS-representation is faithful.
Suppose $a\in A$ has zero kernel and dense range and let $\theta$ be a normal $*$-automorphism of $B$.
Then the following are equivalent:
\begin{enumerate}[label=(\alph*),leftmargin=20pt]
\item\label{it:Eaa*W*pos} $a$ is $B$-valued R-diagonal and for all $k\ge1$ and all $b_1,\ldots,b_{2k-1}\in B$,~\eqref{eq:Eaba*C*pos} holds.
\item\label{it:beta12W*pos} $a$ is $B$-valued R-diagonal and for all $k\ge1$, the $2k$-th order
$B$-valued cumulant maps $\beta_k^{(1)}$ and $\beta_k^{(2)}$
satisfy~\eqref{eq:beta12C*pos}, for all $b_1,\ldots,b_{2k-1}\in B$.
\item\label{it:W*polar} Letting $a=u|a|$ be the polar decomposition of $a$,
\begin{enumerate}[label=(\roman*)]
\item\label{it:uW*} $u$ is a Haar unitary that normalizes $B$ and satisfies $u^*bu=\theta(b)$ for all $b\in B$,
\item\label{it:W*polarfree} $\{u^*,u\}$ and $\{|a|\}$ are free (over $B$) with respect to $\Ect$.
\end{enumerate}
\end{enumerate}
\end{cor}

\begin{example}\label{ex:UB}
Perhaps the easiest example of a $B$-valued R-diagonal element satisfying the conditions of Theorem~\ref{thm:C*pos} is when $p\in B$.
In this case, the operator $a$ can be realized in the crossed product C$^*$-algebra $B\rtimes_\theta\Ints$
with respect to the canonoical conditional expectation onto $B$.
In the case when $B$ is a commutative von Neumann algebra with a specified normal, faithful tracial state and $\theta$ is a trace-preserving
normal automorphism that is ergodic, such $B$-valued R-diagonal operators were studied in \cite{DS09} and, among other results,
their Brown measures were computed.
\end{example}

\section{Algebra-valued circular elements}
\label{sec:circ}

In this section, we examine algebra-valued circular elements, which are a very special class
of algebra-valued R-diagonal elements.

As before,
let $B$ be a unital $*$-algebra and let $(A,\Ec)$ be a $B$-valued $*$-noncommutative probability space;
let $a\in A$; for convenience we label $a_1=a$ and $a_2=a^*$ and use the involution $s:\{1,2\}\to\{1,2\}$ with $s(1)=2$.
Let $\Theta:B\langle X_1,X_2\rangle\to B$ be the $B$-valued $*$-distribution of $(a_1,a_2)$,
where we set $X_1^*=X_2$.
Let
$J=\bigcup_{n\ge1}\{1,2\}^n$ and let $(\alpha_j)_{j\in J}$ be the family of cumulant maps for the pair $(a_1,a_2)$.
\begin{defi}\label{def:circ}
We say that $a$ is {\em $B$-valued circular} if $\alpha_j=0$ whenever $j\in J$ and $j\notin\{(1,2),(2,1)\}$.
\end{defi}

Clearly, $B$-valued circular elements are $B$-valued R-diagonal elements.
The notion of $B$-valued circular first appeared in \cite{Sn03}, where {\'S}niady proved interesting combinatorial results about
a certain $B$-valued circular operator for $B=L^\infty[0,1]$ that he noted was equal to the quasinilpotent DT-opertor $T$ from \cite{DH04a}
(and a proof that it is, in fact, $B$-valued circular can be found in \cite{DH04b}) --- see Example~\ref{ex:circDT} for more details.
The notion of $B$-valued circular operators is closely connected to that of $B$-valued semicircular (also called $B$-Gaussian) operators
that were introduced by Speicher \cite{Sp98} and were studied and constructed on Fock spaces by Shlyakhtenko \cite{Sh99}.

At the purely algebraic level, consider a family $(x_i)_{i\in I}$ of $B$-valued random variables in a $B$-valued noncommutative probability space.
Let $J=\bigcup_{n\ge1}I^n$ and denote by $(\gamma_j)_{j\in J}$ the $B$-valued cumulant maps associated to this family.
The family is said to be {\em centered $B$-valued semicircular} if $\gamma_j=0$ for every $j$ of length not equal to $2$.

In fact, the following result is immediate from the definitions:
\begin{prop}\label{prop:circsemi} 
Suppose $a$ is an element of a $B$-valued $*$-noncommutative probability space and let
\[
x_1=\RealPart a=\frac{a+a^*}2,\quad x_2=\ImagPart a=\frac{a-a^*}{2i}.
\]
Let
$J=\bigcup_{n\ge1}\{1,2\}^n$ and let $(\alpha_j)_{j\in J}$ be the family of cumulant maps for the pair $(a_1,a_2)$.
Let $(\gamma_j)_{j\in J}$ be the family of cumulant maps for the pair $(x_1,x_2)$.
Then the following are equivalent:
\begin{enumerate}[label=(\roman*),labelwidth=3ex,leftmargin=30pt]
\item $a$ is $B$-valued circular
\item the pair $(x_1,x_2)$ is centered $B$-valued semicircular, $\gamma_{(1,1)}=\gamma_{(2,2)}$ and $\gamma_{(1,2)}=-\gamma_{(2,1)}$.
\end{enumerate}
Furthermore, when the above conditions hold, we have
\[ 
\gamma_{(1,1)}=\frac14\left(\alpha_{(1,2)}+\alpha_{(2,1)}\right),\qquad
\gamma_{(1,2)}=\frac i4\left(\alpha_{(1,2)}-\alpha_{(2,1)}\right).
\] 
\end{prop}

Using Speicher's result that freeness is equivalent to vanishing of mixed cumulants (see section 3.3 of \cite{Sp98}), we have:
\begin{cor}
If $a$ is a $B$-valued circular element with associated cumulant maps $\alpha_{(1,2)}$ and $\alpha_{(2,1)}$, then $\RealPart a$
and $\ImagPart a$ are free with respect to the conditional expectation onto $B$ if and only if $\alpha_{(1,2)}=\alpha_{(2,1)}$
\end{cor}

Suppose $B$ is a C$^*$-algebra and $\Theta$ is the $B$-valued distribution of a family of $B$-valued random variables indexed by $I$.
We say that $\Theta$ is positive if $\Theta(p^*p)\ge0$ for every $p\in B\langle X_i\mid i\in I\rangle$.
In general, it can be difficult to decide whether $\Theta$ is positive knowing only the $B$-valued cumulant maps of the family.
However, in the case that the family is $B$-valued semicircular, an answer is provided by Theorem 4.3.1 of \cite{Sp98}.
In the case that $I$ is finite, this condition is equivalent to complete positivity of the covariance $\eta:B\to M_{|I|}(B)$,
which is defined by
\[
\eta(b)=\left(\gamma_{(i_1,i_2)}(b)\right)_{i_1,i_2\in I}.
\]
In the case of the family $(x_1,x_2)$ of real and imaginary parts of a $B$-valued circular system from Proposition~\ref{prop:circsemi},
the covariance
$\eta:B\to M_2(B)=M_2(\Cpx)\otimes B$ is given by
\[
\eta=\frac12\left(\frac12\left(\begin{matrix}1&i\\-i&1\end{matrix}\right)\otimes\alpha_{(1,2)}
+\frac12\left(\begin{matrix}1&-i\\i&1\end{matrix}\right)\otimes\alpha_{(2,1)}\right).
\]
Thus, we have the following corollary of Speicher's theorem mentioned above
and of Shly\-akh\-tenko's construction~\cite{Sh99} of $B$-valued semicircular elements.

To recall the set-up,
let $B$ be a $C^*$-algebra as before,
consider a $B$-valued circular element $a$, let $(a_1,a_2)=(a,a^*)$ and let $\Theta$ be the distribution and $\alpha_{(1,2)}$, $\alpha_{(2,1)}$
the cumulant maps of $(a_1,a_2)$, as usual.
Endow the algebra $B\langle X_1,X_2\rangle$ with the $*$-operation that extends the given one on $B$ by setting $X_1^*=X_2$.
\begin{cor}\label{cor:poscirc}
The distribution $\Theta$ is positive if and only if $\alpha_{(1,2)}$ and $\alpha_{(2,1)}$ are completely positive maps from $B$ into itself.
In this case, $\Theta$ is the $*$-distribution of a $B$-valued circular element
in a $B$-valued C$^*$-noncommutative probability space.
\end{cor}

Turning back to the case of $B$ a general $*$-algebra, without assuming anything about boundedness or positivity,
from Proposition~\ref{prop:traceRdiag}, we immediately see the following:
\begin{prop}\label{prop:traceCirc}
If $a$ is a $B$-valued circular element with associated cumulant maps $\alpha_{(1,2)}$ and $\alpha_{(2,1)}$ and if
$\tau$ is a tracial linear functional on $B$, then $\tau\circ\Ec$ is tracial on $\alg(B\cup\{a,a^*\})$
if and only if, for all $b_1,b_2\in B$, we have
\[
\tau(\alpha_{(1,2)}(b_1)b_2)=\tau(b_1\alpha_{(2,1)}(b_2)).
\]
\end{prop}

We will now examine the power series considered in Section~\ref{sec:series}
for a $B$-valued circular element $a$.
Theorem~\ref{thm:powerseries} yields
\begin{prop}\label{prop:recur}
Consider the formal power series (in the sense described in Section~\ref{sec:series})
\begin{align}
F(b_1,b_2)&=\sum_{n=0}^\infty\Ec\big((ab_1a^*b_2)^n\big) \label{eq:fb1b2} \\
G(b_1,b_2)&=\sum_{n=0}^\infty\Ec\big((a^*b_1ab_2)^n\big). \label{eq:gb1b2}
\end{align}
for a $B$-valued circular element $a$.
Then
\begin{align}
F(b_1,b_2)&=1+\alpha_{(1,2)}(b_1G(b_2,b_1))b_2F(b_1,b_2), \label{eq:frecur} \\[0.5ex]
G(b_1,b_2)&=1+\alpha_{(2,1)}(b_1F(b_2,b_1))b_2G(b_1,b_2).  \label{eq:grecur} 
\end{align}
\end{prop}

In the case of a $B$-valued C$^*$-noncommutative probability space, the series~\eqref{eq:fb1b2}--\eqref{eq:gb1b2} 
define $B$-valued holomorphic functions
with domain equal to the subset of $B\oplus B$ consisting of all pairs $(b_1,b_2)$ such that $\|b_1\|\,\|b_2\|<\|a\|^2$.
From these functions one can recover only some of the information about the $B$-valued distribution of $a$.
Later, in the appendix, we will study a particular case of these functions evaluated at $b_1,b_2\in\Cpx1$.

\begin{example}\label{ex:circDT}
As observed by {\'S}niady \cite{Sn03}, (see \cite{DH04b} for a proof), the quasinilpotent DT-operator $T$ from \cite{DH04a}
is a $L^\infty[0,1]$-valued circular operator, with cumulant maps given by
\[ 
\alpha_{(1,2)}(f)(x)=\int_x^1f(t)\,dt,\qquad
\alpha_{(2,1)}(f)(x)=\int_0^xf(t)\,dt.
\] 
We note that if $\theta$ is the automorphism of $L^\infty[0,1]$ corresponding to the homeomophism $t\mapsto 1-t$ of $[0,1]$,
then
\[
\alpha_{(2,1)}(b)=\theta(\alpha_{(1,2)}(\theta(b))).
\]
Thus, Corollary~\ref{cor:polarfree}
applies and we have the following result.
\end{example}

\begin{cor}\label{cor:DTpolar}
Let $T$ be a quasinilpotent DT-operator in an $L^\infty[0,1]$-valued $W^*$-non\-com\-mu\-ta\-tive probability space $(A,\Ec)$
with $\Ec$ faithful
and let $T=U|T|$ be its polar decomposition.
Then
\begin{enumerate}[label=(\roman*),labelwidth=3ex,leftmargin=30pt]
\item $U$ is an $L^\infty[0,1]$-normalizing Haar unitary element,
\item for every $b\in L^\infty[0,1]$ we have $U^*bU=\theta(b)$,
\item $\{U^*,U\}$ and $\{|T|\}$ are free with respect to $\Ec$.
\end{enumerate}
\end{cor}

As advertised, here is an application of Lemma~\ref{lem:notfree},
yielding a $B$-valued R-diagonal element $a$ with $B$ two-dimensional and such that in the polar decomposition
$a=u|a|$ of $a$, $u$ is unitary but $\{u,u^*\}$ and $\{|a|\}$ are not free over $B$.
\begin{example}\label{ex:nofreepolar}
Take $B=\Cpx\oplus\Cpx$, $\tau_B(\lambda_1\oplus\lambda_2)=\frac{\lambda_1+\lambda_2}2$ and define 
\begin{align*}
\alpha_{(1,2)}(\lambda_1\oplus\lambda_2)&=\frac{\lambda_1}2\oplus\left(\frac{\lambda_1}2+\lambda_2\right) \\
\alpha_{(2,1)}(\lambda_1\oplus\lambda_2)&=\frac{\lambda_1+\lambda_2}2\oplus\lambda_2.
\end{align*}
Then $\alpha_{(1,2)}$ and $\alpha_{(2,1)}$ are completely positive and
satisfy the condition of Proposition~\ref{prop:traceCirc} for traciality.
Thus, by Corollary~\ref{cor:poscirc}, there is a $B$-valued C$^*$-noncommutative probability space $(A,\Ec)$
containing a $B$-valued circular element $a$ with corresponding cumulant maps $\alpha_{(1,2)}$ and $\alpha_{(2,1)}$;
we assume $A$ is generated as a C$^*$-algebra by $B\cup\{a\}$, and that the GNS representation
of $\Ec$ is faithful;
by Proposition~\ref{prop:traceCirc}, $\tau=\tau_B\circ\Ec$ is a positive trace on $A$;
since it has faithful GNS representation, it is faithful.
We also have $\alpha_{(2,1)}(1)=1$ and $\alpha_{(1,2)}(1)\ne1$.
Thus by Lemma~\ref{lem:notfree},
$a$ does not have the same $B$-valued $*$-distribution
as any element $up$ in any $B$-valued $*$-noncommutative probability space with $u$ unitary,
with $p$ self-adjoint and with $\{u,u^*\}$ and $\{p\}$ free over $B$.
In particular, we may take the von Neumann algebra generated by the image of the GNS representation of $\tau$
we get a larger $B$-valued $*$-noncommutative probability space $(\At,\Ect)$ and in $\At$ we have
the polar decomposition $a=u|a|$ of $a$.
By Proposition~\ref{prop:afacts} below, $a$ has zero kernel, so $u$ is unitary.
By Proposition~\ref{prop:W*polar}, $u$ is a Haar unitary.
But $\{u,u^*\}$ and $\{|a|\}$ cannot be free with respect to $\Ect$ (from Lemma~\ref{lem:notfree}, as remarked above).
\end{example}

\appendix

\section{On a distribution}
\label{sec:dist}

In this appendix, we describe investigations of the distribution $\mu_{a^*a}=\mu_{aa^*}$ of the element $aa^*$
with respect to the trace $\tau$,
where $a$ is the $B$-valued circular element described in Example~\ref{ex:nofreepolar}.
A Mathematica Notebook file containing the detailed calculations will be made available with this paper.\footnote{
The Notebook file is available with the arXiv source, see  arXiv:1512.06321, or from
http://www.math.tamu.edu/$\tilde{\;}\,$kdykema/Research/rdiag.NonFreePolarExample.nb.}

Using the equations~\eqref{eq:fb1b2} and~\eqref{eq:gb1b2} with $b_1=b_2=s1$ for $s$ in some neighborhood of $0$ in $\Cpx$
and letting $z=s^2$, we have
\begin{align*}
f(z)&:=F(s1,s1)=\sum_{n=0}^\infty \Ec((aa^*)^n)z^n \\
g(z)&:=G(s1,s1)=\sum_{n=0}^\infty \Ec((a^*a)^n)z^n.
\end{align*}
We are interested in the moment generating function for $aa^*$ with respect to the trace $\tau$, and this is the function
\[
h(z):=\tau(f(z))=\sum_{n=0}^\infty\tau((aa^*)^n)z^n=\sum_{n=0}^\infty\tau((a^*a)^n)z^n=\tau(g(z)).
\]
The recursive relations in Proposition~\ref{prop:recur} yield
\begin{align*}
f(z)&=1+z\,\alpha_{(1,2)}(g(z))\,f(z) \\
g(z)&=1+z\,\alpha_{(2,1)}(f(z))\,g(z).
\end{align*}
Since $B=\Cpx\oplus\Cpx$ and $f$ and $g$ are $B$-valued functions, we write $f=f_1\oplus f_2$ and $g=g_1\oplus g_2$.
Using the definitions of $\alpha_{(1,2)}$ and $\alpha_{(2,1)}$ from Example~\ref{ex:nofreepolar}, we have the recursive relations
\begin{align} 
f_1&=1+z\left(\frac{g_1}2\right)f_1 \\
f_2&=1+z\left(\frac{g_1}2+g_2\right)f_2 \\
g_1&=1+z\left(\frac{f_1+f_2}2\right)g_1 \\
g_2&=1+zf_2g_2.
\end{align}
We substitute $h=(f_1+f_2)/2$ and, using simple elimination, arrive at a polynomial identity for $h$:
\begin{equation}\label{eq:hpoly}
8z^3h^4-20z^2h^3+8z(z+2)h^2+(z^2-12 z-4)h+4=0.
\end{equation}
This allows computation of arbitrarily many terms of the series expansion for $h$ around $0$, and we find
\begin{equation}\label{eq:hseries}
h(z)=1+z+\frac{9}{4}z^2+\frac{13}{2}z^3+\frac{341}{16}z^4+\frac{1207}{16}z^5+\frac{17985}{64}z^6+O(|z|^7).
\end{equation}

The Stieltjes transform $G=G_{\mu_{a^*a}}$ of the measure $\mu_{a^*a}$ is, for $w$ in the complement of the closed support of $\mu_{a^*a}$,
\[
G(w)=\int_{\Reals}\frac1{w-t}\,d\mu_{a^*a}(t)=w^{-1}h(w^{-1}).
\]
Note that $G(w)$ is real when $w$ is real and $|w|$ is large.
From~\eqref{eq:hpoly} we get the identity
\begin{equation}\label{eq:Gpoly}
8G^4w^2-20G^3w^2+8G^2w(2w+1)+G(-4w^2-12w+1)+4w=0.
\end{equation}
Stieltjes inversion can be used to find the measure $\mu_{a^*a}$ from knowledge of the algebraic function $G$:
$\mu_{a^*a}$ has an atom at a point $t_0\in\Reals$ if and only if,
\begin{equation}\label{eq:Gatom}
\liminf_{\eps\to0^+}\eps G(t_0+\eps i)>0,
\end{equation}
and then the limit exists and the value of this limit equals $\mu_{a^*a}(\{t_0\})$
while elsewhere on the real line, $\mu_{a^*a}$ has density given by
\[
\frac{d\mu_{a^*a}}{dt}(t)=\lim_{\eps\to0^+}\frac{-\ImagPart G(t+i\eps)}\pi.
\]
Note that, since $a^*a$ is positive and bounded, $\mu_{a^*a}$ is compactly supported in $[0,\infty)$.

Of course, an explicit formula for $G$ can be found involving radicals, (solving the quartic)
 and care must be taken to find the correct branch
of $G$ near the real axis;
the correct branch of $G$ is the one that near $w=\infty$ has asymptotics $G(w)=\frac1w+O(|w|^{-2})$.
However, the following facts can be seen by a somewhat less arduous analysis:
\begin{prop}\label{prop:afacts}
The element $a$ has zero kernel but is not invertible.
Its norm is the root of the polynomial
\[
16 x^8-160 x^6+540 x^4-680 x^2+27
\]
that is approximately equal to $2.18942$.
\end{prop}
\begin{proof}
Using the method of Newton's polytope to find Puiseux series asymptotic expansions
for the algebraic function roots of the polynomial equation~\eqref{eq:Gpoly}, we find that near $w=0$, the four algebraic functions are
\begin{align}
G_1(w)&=-4w-48w^2+O(|w|^3) \label{eq:G1} \\
G_j(w)&=-\frac12w^{-2/3}+\frac23w^{-1/3}+\frac56+O(|w|^{1/3})\quad (j=2,3,4) \label{eq:G2}
\end{align}
where the three roots $G_j$ for $j=2,3,4$ result from chosing the different branches of $w^{1/3}$.
We set $G_2$ to be the one that makes $w^{1/3}$ negative real when $w$ is negative real.
Already we see that the Stieltjes transform $G$ fails the condition~\eqref{eq:Gatom} at $t_0=0$, 
so $\mu_{aa^*}=\mu_{a^*a}$ has no atom at $0$ and,
thus, $a$ has zero kernel and zero co-kernel.

Since $\mu_{a^*a}$ has no support in $(-\infty,0)$,
we see that the Stieltjes transform must be equal to either $G_1$ or $G_2$ near $w=0$ (and in the domain of $G$),
since as $w$ approaches negative numbers near zero from above, both $\ImagPart G_3$ and $\ImagPart G_4$ approach nonzero numbers.
To see that $a$ is not invertible, it will suffice to see that $G=G_2$, since it will yield nonzero density for $\mu_{a^*a}$
on some interval $[0,\delta)$.

The discriminant of the polynomial in~\eqref{eq:Gpoly}, with respect to the variable $G$, is
\begin{equation}\label{eq:discr}
-64w^4(16 w^4-160 w^3+540 w^2-680 w+27),
\end{equation}
whose real roots, other than $0$, are approximately $0.0410263$, and $4.79356$.
The coefficient of $G^4$ in~\eqref{eq:Gpoly} is $w^2$, which has roots only at $0$.
Thus, the only place on the real axis where $G$ can diverge to $\infty$ is $0$,
and the only places where two or more of the four algebraic roots of~\eqref{eq:Gpoly}
can agree is where $w$ is one of the roots of the polynomial~\eqref{eq:discr}.
The polynomial~\eqref{eq:Gpoly} has real coefficients, so roots come in complex conjugate pairs.
Thus $G(w)$, which is real and negative for $w<<0$, must remain real as $w$ approaches $0$ through negative numbers.
Moreover, setting $G$ equal to $0$ in~\eqref{eq:Gpoly} yields $w=0$, so $G$ is nonvanishing on the negative real axis.
Therefore,  $G(w)$ is real and strictly negative for all $w\in(-\infty,0)$.
However, this is clearly not possible for the function $G_1$, according to the asymptotics~\eqref{eq:G1}.
Thus, $G=G_2$ and $a$ is not invertible.

Similar considerations show that the maximum of the support of $\mu_{a^*a}$,
which is equal to $\|a\|^2$,
can only be one of the roots of~\eqref{eq:discr}.
However, the coefficients in the moment series~\eqref{eq:hseries} yield lower bounds on $\|a\|^2$;
already the coefficient of $z$ yields
$\|a\|^2\ge\tau((aa^*))=1$.
So $\|a\|^2$ equals the largest of the real roots of~\eqref{eq:discr}, as required for the assertion in the proposition about $\|a\|$.
\end{proof}

\begin{remark}
From the asymptotics~\eqref{eq:G2}, we easily see that
the density of the measure $\mu_{a^*a}$ behaves asymptotically like
\[
\frac{d\mu_{a^*a}}{dt}(t)=\frac{\sqrt3}{4\pi}t^{-2/3}+O(t^{-1/3})
\]
as $t$ approaches $0$ from the right.
Thus, the distribution $\mu_{|a|}$ has density whose asymptotic expansion is
\[
\frac{d\mu_{|a|}}{ds}(s)=2s\frac{d\mu_{a^*a}}{dt}(s^2)=\frac{\sqrt3}{2\pi}s^{-1/3}+O(s^{1/3})
\]
as $s$ tends to $0$ from the right.
We calculated the density of $\mu_{|a|}$ numerically; a plot is in Figure~\ref{fig:densplot};
for comparison purposes, the density $\frac{d\mu_{|z|}}{ds}(s)=\frac1\pi\sqrt{4-s^2}$
of the quarter-circular element $|z|$ is plotted on the same grid, with the dashed line.
For details see the Mathematica Notebook file that is available with this paper.
\begin{figure}[hb]
\caption{The density of the measure $\mu_{|a|}$.}
\label{fig:densplot}
\includegraphics{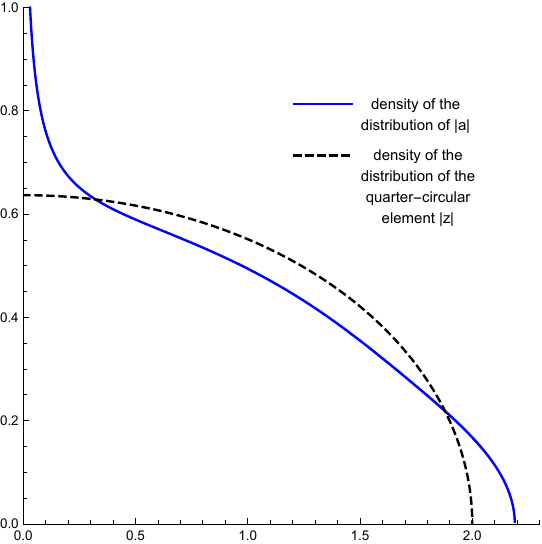}
\end{figure}
\end{remark}

\section{Erratum (A version of which will appear in print in the Houston Journal of Mathematics)}

Theorem~5.8 of this paper is incorrect.
While assertions~(a) and~(b) are equivalent and the implication (c)$\implies$(a) holds, the implication (a)$\implies$(c) fails to hold.
A counter example is provided by taking $a$ to be a circular element in a tracial (scalar-valued) $C^*$-noncommutative probability space $(A_0,\tau_0)$, normalized so that $\tau_0(a^*a)=1$.
We assume in addition that $a$ has polar decomposition $a=v|a|$ for a unitary $v\in A_0$.
Then, as is well known, $\tau_0(v^n)=0$ for all integers $n\ge1$.
Taking $B$ to be a unital C$^*$-algebra, $B\ne\Cpx$, with a faithful tracial state $\tau_B$,
we let $(A,\tau)=(A_0,\tau_0)*(B,\tau_B)$ be the free product with amalgamation over the scalars and let $\Ec:A\to B$ be the $\tau$-preserving conditional expectation.
By Theorem~6 of~\cite{SS01}, the element $a$ is also $B$-valued circular with respect to $\Ec$, with cumulant maps (in the notation of Theorem~5.8) given by
\[
\beta_1^{(1)}(b)=\beta_1^{(2)}(b)=\tau_B(b)1
\]
and $\beta_k^{(1)}=\beta_k^{(2)}=0$ for all $k\ge2$.
Thus, the hypothesis of part~(b) of Theorem~5.8 is fulfilled, when $\theta$ is any $\tau_B$-preserving automorphism of $B$, in particular when $\theta$ is the identity automorphism.
However, $a$ cannot have the same $B$-valued distribution as a product $up$ for $p\ge0$ and $u$ a unitary commuting with $B$, so part~(c) of the theorem fails to hold in this case.
In fact, for such a product $up$, the $B$-valued $*$-distribution of $u$ is determined by that of $a$.
(This can be seen, for example, by representing everything on a Hilbert space and choosing polynomials $q_n$ so that $up\,q_n(p^2)$ stays bounded and converges in strong operator topology to $u$ as $n\to\infty$.)
However, we already have $a=v|a|$ where $v$ is $*$-free from $B$ and is not a scalar multiple of the identity, so cannot commute with $B$.

The purported proof of Theorem~5.8 is in error on page 242, line 8 of the published version,
when it asserts that the pair $(u,p)$ has the same $B$-valued $*$-distribution as $(u,s|p|)$.
In particular, $p$ and $s|p|$ need not have the same $B$-valued $*$-distribution.
Indeed, while we do have $\Et((s|p|)^n)=\Et(p^n)$ for every integer $n\ge1$, we need not have, for example, $\Et(s|p|bs|p|)=\Et(pbp)$ for all $b\in B$.

When $a$ commutes with every element of $B$, then the validity of the proof and result hold.
Thus, the application in the proof of the (already well known) Corollary~5.9 is valid.
However, Corollary~5.10 is incorrect, as shown by the example given above.
Finally, Corollary~6.8 is not proved and may well be incorrect.

\end{document}